\documentclass[11pt]{amsart}

\usepackage[dvipsnames,svgnames,x11names]{xcolor}
\usepackage{textgreek,bbm,url,graphicx,verbatim,amssymb,enumerate,stmaryrd,booktabs,lmodern,mathtools,microtype,mathabx}
\usepackage[pagebackref,colorlinks,citecolor=Mahogany,linkcolor=Mahogany,urlcolor=Mahogany,filecolor=Mahogany]{hyperref}
\usepackage[capitalize]{cleveref}
\usepackage[mathscr]{euscript}
\usepackage{tikz,tikz-cd}
\usetikzlibrary{matrix,calc,positioning,arrows,decorations.pathreplacing,patterns,arrows}
\usepackage{mathabx}
\usepackage{fancyhdr}
\usepackage[a4paper,
            bindingoffset=0.2in,
            left=0.8in,
            right=1in,
            top=1in,
            bottom=1in,
            footskip=.25in]{geometry}   
\usepackage{enumitem}
\usepackage{bm}
\makeatletter
\DeclareFontFamily{OMX}{MnSymbolE}{}
\DeclareSymbolFont{MnLargeSymbols}{OMX}{MnSymbolE}{m}{n}
\SetSymbolFont{MnLargeSymbols}{bold}{OMX}{MnSymbolE}{b}{n}
\DeclareFontShape{OMX}{MnSymbolE}{m}{n}{
    <-6>  MnSymbolE5
   <6-7>  MnSymbolE6
   <7-8>  MnSymbolE7
   <8-9>  MnSymbolE8
   <9-10> MnSymbolE9
  <10-12> MnSymbolE10
  <12->   MnSymbolE12
}{}
\DeclareFontShape{OMX}{MnSymbolE}{b}{n}{
    <-6>  MnSymbolE-Bold5
   <6-7>  MnSymbolE-Bold6
   <7-8>  MnSymbolE-Bold7
   <8-9>  MnSymbolE-Bold8
   <9-10> MnSymbolE-Bold9
  <10-12> MnSymbolE-Bold10
  <12->   MnSymbolE-Bold12
}{}

\let\llangle\@undefined
\let\rrangle\@undefined
\DeclareMathDelimiter{\llangle}{\mathopen}%
                     {MnLargeSymbols}{'164}{MnLargeSymbols}{'164}
\DeclareMathDelimiter{\rrangle}{\mathclose}%
                     {MnLargeSymbols}{'171}{MnLargeSymbols}{'171}
\makeatother

\def\Xint#1{\mathchoice
{\XXint\displaystyle\textstyle{#1}}%
{\XXint\textstyle\scriptstyle{#1}}%
{\XXint\scriptstyle\scriptscriptstyle{#1}}%
{\XXint\scriptscriptstyle\scriptscriptstyle{#1}}%
\!\int}
\def\XXint#1#2#3{{\setbox0=\hbox{$#1{#2#3}{\int}$ }
\vcenter{\hbox{$#2#3$ }}\kern-.6\wd0}}

\def\dashint{\Xint-}

\newtheorem{theorem}{Theorem}[section]
\newtheorem{lemma}[theorem]{Lemma}
\newtheorem{proposition}[theorem]{Proposition}
\newtheorem{corollary}[theorem]{Corollary}
\newtheorem*{corollary*}{Corollary}

\newtheorem{atheorem}{Theorem}

\theoremstyle{definition}
\newtheorem{definition}[theorem]{Definition}

\theoremstyle{remark}
\newtheorem{example}[theorem]{Example}

\newtheorem{remark}[theorem]{Remark}

\newcommand{\subref}[2]{\hyperref[#2]{\ref{#1}.\ref{#2}}}

\numberwithin{equation}{section}
\crefname{equation}{}{}

\newcommand{\sq}[1]{\widetilde{#1}}
\newcommand{\R}{\mathbb{R}}

\newcommand{\N}{\mathbb{N}}

\newcommand{\mc}[1]{\mathcal{#1}}
\newcommand{\fl}[1]{\lfloor #1 \rfloor}

\newcommand{\ang}[1]{\langle #1 \rangle}
\newcommand{\dang}[1]{\llangle #1 \rrangle}
\newcommand{\D}{\mathcal{D}}
\newcommand{\dd}{\mathrm{dist}_\D}

\newcommand{\1}{\mathbf{1}}
\newcommand{\LL}{\langle}
\newcommand{\RR}{\rangle}
\newcommand{\BMO}{\mathrm{BMO}}

\newcommand{\bv}[1]{\bm{|} #1 \bm{|}}

\title{Matrix weighted $L^p$ estimates in the nonhomogeneous setting}

\author{Fernando Benito-de la Cigoña, Tainara Borges, Francesco D'Emilio, Marcus Pasquariello, 
 Nathan A. Wagner}

\begin{document}
\begin{abstract}
We 
establish a modified pointwise convex body domination for vector-valued Haar shifts in the nonhomogeneous setting, strengthening and extending the scalar case developed in \cite{CPW}. Moreover, we identify a subclass of shifts, called $L^1$-normalized, for which the standard convex body domination holds without requiring any regularity assumption on the measure. Finally, we extend the best-known matrix weighted $L^p$ estimates for sparse forms to the nonhomogeneous setting. The key difficulty here is the lack of a reverse-Hölder inequality for scalar weights, which was used in \cite{cruzuribe2017} to establish $L^p$ matrix weighted estimates and only works in the doubling setting. 
Our approach relies instead on a generalization of the weighted Carleson embedding theorem which allows to control not only a fixed weight, but also collections of weights localized on different dyadic cubes that satisfy a certain compatibility condition.
\end{abstract}

\thanks{\textbf{Funding Acknowledgment:} Francesco D'Emilio was partially supported by National Science Foundation DMS-2054863 and by the Agencia Estatal de Investigación (Spain), grant number CNS2022-135431. Nathan A. Wagner was supported by a National Science Foundation MSPRF DMS-2203272. }
\maketitle

\section{Introduction and main results}
In this paper, we seek to further understand the role of dyadic harmonic analysis in the context of nonhomogeneous measures in $\R^n$. A measure $\mu$ in $\R^n$ is said to be doubling if $$\mu(B(x,2r))\leq C\mu(B(x,r)), \quad \text{ for all $r>0$ and $x\in \R^n$.} $$ In the doubling setting, two different perspectives to understand Calderón-Zygmund operators through the lenses of dyadic analysis have been developed in the last two decades. The first approach relies on representation theorems, which assert that Calderón-Zygmund operators can be weakly represented as probabilistic averages of dyadic shifts and paraproducts. The second approach is based on sparse domination, which uses a self-similarity argument to provide pointwise domination of these operators by a simpler class of positive dyadic operators known as sparse operators. Both approaches have independent interest: the former provides a modern proof to $T(1)$ theorems and allowed to obtain sharp weighted inequalities through a complex reduction to dyadic shifts and paraproducts, see \cite{Petermichl, PetermichlVolberg, HytonenAnnals}. On the other hand, sparse domination has greatly simplified the proof of weighted inequalities, leveraging the relatively elementary nature of sparse operators \cite{Lerner2013, Moen}.

In the nonhomogeneous setting, i.e, when one does not assume the doubling condition on the measure, Calder\'on-Zygmund theory was developed in $\R^n$ as long as the measure satisfied a power growth condition \cite{NTV97,NTV98, NTV2003, TolsaRBMO,TolsaT1}
$$\mu(B(x,r)) \lesssim_\mu r^d, \quad \text{ for all } x \in \R^n, \ r>0, \text{ for some } 0<d \leq n.$$ 
While both approaches—representation and sparse domination—have proven to be extremely 
powerful in the doubling setting, they encounter significant limitations in the nonhomogeneous context. In particular, when applied to weighted inequalities, the representation theorem of \cite{HH} gives rise to certain dyadic operators that are difficult to manage. Similarly, existing sparse domination techniques fail to recover the sharp $A_2$ theorem or even the correct weight class, highlighting gaps that warrant further investigation \cite{Conde-AlonsoParcet, VolbergZorin-Kranich, Tolsa2007}.

Back to the heart of the matter, in \cite{LSMP} the authors introduced sufficient conditions on a general measure $\mu$ to get bounds for scalar valued Haar shifts defined with respect to some fixed dyadic grid $\D$ (see Definition \ref{def: vectorhaarshift} with $d=1$ for the definition of a scalar Haar shift). In the real line, they characterized the class of measures for which every scalar Haar shift is weak type $(1,1)$ (consequently such condition was also enough for any scalar Haar shift to be bounded in $L^{p}(\mu)$, $1<p<\infty$). Such measures, later termed \emph{balanced}, are the ones satisfying that for every dyadic interval $I \in \D$,
\begin{equation}\label{balanced}
    m(I) \sim m(\widehat{I}), \quad \text{for} \quad m(I):= \frac{\mu(I_+)\mu(I_-)}{\mu(I)},
\end{equation}
where $\widehat{I}$ is the dyadic father of $I$ in $\D$.
Interestingly, there seems to be no direct relationship between the power growth condition and the balanced condition. This observation reinforces a broader theme: an apparent discrepancy between continuous and dyadic analysis in the nonhomogeneous setting, which is also reflected in the limitations of current representation and sparse domination techniques.

A renewed interest in the topic arose after the work of J. Conde Alonso, J. Pipher, and the last author \cite{CPW}, where it was shown that even in the context of balanced measures, sparse domination in the usual sense fails for the dyadic Hilbert transform and more general dyadic shifts. To prove sparse domination for such Haar shifts, they had to consider more general sparse forms that encompassed the complexity of the dyadic shift of interest, and which were still good enough to deduce weighted corollaries for the right class of adapted $A_p$ weights.

Our first result improves the state of the art in several ways: we properly extend to higher dimensions $n\geq 1$ the balanced condition for an arbitrary measure $\mu$ and standard Haar systems $\mathscr{H}$ in $\R^n$, which is the right condition to have continuity in $L^{p}(\mu)$ for any scalar Haar shift of arbitrary complexity. For such balanced pairs (see \cref{def of balanced pairs}) we will upgrade the modified sparse domination to a pointwise sparse domination, which is stronger than the dual form in which it originally appeared in \cite{CPW}, and we will prove that in the more general context of vector Haar shifts. Moreover, we determine a class of shifts for which the classical sparse domination holds without any assumption on the measure. In doing so, we extend to the nonhomogeneous setting the application of convex body domination, which is the natural analog of scalar sparse domination in the vector-valued setting and was introduced in \cite{NPTV} to prove matrix weighted inequalities. \par Given a vector valued function $f: \R^n \to \R^d$ which is $\mu$-integrable on a cube $Q$ in $\R^d$, the convex body average $\dang{f}_Q$ is the compact, convex and symmetric set defined as the image of the unit ball of $L^\infty(Q)$ under the bounded linear functional defined by the pairing with $f$
\[\dang{f}_Q := \{ \LL f \psi \RR_Q, \ \psi: Q \to \R, \|\psi \|_\infty \leq 1 \},\]
 where 
 \[ \LL f \psi \RR_Q := \frac{1}{\mu(Q)} \int_Q f(x) \psi(x) d\mu(x),\] 
is the vector whose $i$-th component is $\langle f_i \psi \rangle_Q$, for $i=1, \dots, d.$ If $T$ is a linear operator acting on scalar valued functions and $f: \R^n \to \R^d$, we abuse notation writing $Tf$ instead of $(T \otimes I_d)(f)$, where $(T \otimes I_d)(f)=(Tf_1, \dots, Tf_d)$. In the rest of the paper, for $1<p<\infty$ the notation $L^p(\R^n, \R^d)$ or simply $L^p$ will generally refer to the usual $L^p$-space of (vector valued) functions with respect to an arbitrary measure $\mu$, unless otherwise specified. We are now ready to state our first main result.

\begin{atheorem}\label{thm: sparseforvectorhaar}
Let $\mu$ a Radon measure in $\R^n$ and $\mathscr{H}$ a generalized Haar system such that the pair $(\mu, \mathscr{H})$ is balanced. Let $f \in L^1(\R^n;\R^d)$ be compactly supported in $Q_0 \in \D$, and $T$ be a vector valued Haar shift of complexity $(s,t)$ as in Definition \ref{def: vectorhaarshift}. There exists a sparse family $\mathcal{S}= \mathcal{S}(f)\subset \mathcal{D}(Q_0)$ and a positive constant $C=C(n,d,s,t,T)$, depending exponentially on the complexity, such that  
\[Tf(x) \in C \sum_{Q \in \mc{S}}\dang{f}_Q \1_{Q}(x) + C\sum_{\substack{ J, K \in \mc{S} \\ \text{dist}(J, K) \leq s+t + 2 }} \dang{f}_{J} \frac{\1_K(x)}{\mu(K)}\sqrt{m(J)}\sqrt{m(K)}  \quad \text{ on } Q_0.\]   

Moreover, for general Radon measures $\mu$ and Haar systems $\mathscr{H}$, if $T$ is $L^1$ normalized as in \cref{D: L^1 normalized}, we obtain the usual convex body domination. Namely, for each $f \in L^1(\R^n;\R^d)$ compactly supported in $Q_0 \in \D$, there exists a sparse family $\mathcal{S}= \mathcal{S}(f)\subset \mathcal{D}(Q_0)$ and a positive constant $C=C(n,d,s,t,T)$ depending linearly on the complexity such that

\[Tf(x) \in C\sum_{Q \in \mc{S}}\dang{f}_Q \1_{Q}(x) \quad \text{ on } Q_0. \]
 
\end{atheorem}
It is important to note that, in the nonhomogeneous setting, the assumption $x \in Q_0$ can not be extended to hold for almost every $x \in \R^n$, as is the case in the doubling setting (see also \cite{Lacey_2017} where a similar restriction is imposed on sparse domination for martingale transforms with general measures). However, the resulting conclusion remains sufficient to establish weighted estimates, as we shall see.\\

In what follows, we always assume $W$ is a $d \times d$ matrix and a function $f$ takes values in $\R^d$, even if it is not explicitly included in the notation. There has been persistent interest in matrix weighted inequalities since the matrix $A_2$ condition was introduced in \cite{TreilVolberg} to characterize $L^2(W)$ bounds for the vector-valued Hilbert transform. This problem was initially motivated by the study of random multivariate stationary processes. Let $\mathcal{M}_{d \times d}^{+}$ denote the space of $d \times d$ positive semi-definite matrices with real entries. A matrix weight $W$, that is a locally integrable function $W: \R^n \rightarrow \mathcal{M}_{d \times d}^{+}$, satisfies the matrix $A_2$ condition if

\begin{equation}
[W]_{A_2}:= \sup_{Q}\, \bm{|} \langle W \rangle_{Q}^{1/2} \, \langle W^{-1} \rangle_Q^{1/2} \bm{|}^2< \infty.   
\end{equation}

The theory of matrix weighted inequalities was later expanded to include $L^p(W)$ estimates for exponents $p \neq 2$, as well as matrix weighted estimates for general Calder\'{o}n-Zygmund operators, weighted estimates for maximal operators, two weight bump estimates, and other innovations \cite{ChristGoldberg, cruzuribe2017, Volberg97}. The significant interest in sharp weighted norm inequalities in the scalar setting has naturally led to the consideration of the analogous problem in the matrix case. After several partial improvements (see, for example \cite{BickelPetermichlWick}), the dependence of the $L^2(W)$ operator norm on the weight characteristic was shown to be no worse than $[W]_{A_2}^{3/2}$ in \cite{NPTV}, using convex body domination. These results led to the formulation of the matrix $A_2$ conjecture, which states that the operator norm of the Hilbert transform on $L^2(W)$ should depend linearly in $[W]_{A_2};$ as in the scalar case. Quite surprisingly, the matrix $A_2$ conjecture was shown to be false for the Hilbert transform in the recent paper \cite{DPTV}: the exponent $3/2$ is actually sharp. Also, the question of the optimal weight dependence for $p \neq 2$ remains open in the homogeneous setting, even in the specific case of the Hilbert transform. To date, the best known $L^p(W)$ norm estimate for Calder\'{o}n-Zygmund operators was given in \cite{cruzuribe2017}, where the authors used two weight Orlicz bump technology to deduce one weight corollaries. The following bound is proven in that paper
    \[\|T\|_{L^p(W) \to L^p(W)} \lesssim [W]_{A_p}^{1 + \frac{1}{p-1} - \frac{1}{p}},\]
where $T$ is a classical Calder\'on-Zygmund operator. We emphasize that all of these mentioned contributions were in the \emph{homogeneous setting} (typically presented with Lebesgue measure).

In the nonhomogeneous setting, there has been less activity concerning matrix weighted inequalities. However, we mention the work of Treil in this area \cite{Treil2023}, specifically in estimates obtained for a vector square function with matrix weights where the underlying measure does not satisfy any doubling condition. See also \cite{CuliucTreil} for a version of the Carleson embedding theorem in the matrix, nonhomogeneous setting.

Our second main result recovers the best known dependence on $[W]_{A_p}$ in the doubling case, does not make use of Orlicz bump functions, and works much more generally for any Radon measure $\mu$ for which we have a certain bilinear sparse form domination of an operator acting on vector-valued functions. This dependence is sharp in the case $p=2.$ We mention that an additional advantage of the weighted estimate is that, a priori, it does not make any reference to technical constructions involving convex bodies (it is purely an estimate on a very natural bilinear form in the vector valued case). 

\begin{atheorem} \label{MatrixWeightedThm}
Let $\mu$ be a Radon measure on $\R^n$ and let $\mc{S}$ be a $\eta$-sparse family of cubes. Consider the sparse form 
    \[\mc{A}_{\mc{S}}(f, g) := \sum_{Q \in \mc{S}}\dashint_{Q}\dashint_{Q}|f(x) \cdot g(y)| \, d\mu(x) \, d\mu(y)\mu(Q).\]
    Then for $1 < p < \infty$, if $W \in A_{p, \infty}^{sc}$, $V \in A_{p', \infty}^{sc}$ and $[W, V]_{A_p} < \infty$, we have
    \[\mc{A}_{\mc{S}}(V^{1/p'}f, W^{1/p}g) \lesssim_{d, \eta} [W]_{A_{p, \infty}^{sc}}^{1/p'} [W, V]_{A_p}^{1/p}[V]_{A_{p', \infty}^{sc}}^{1/p}\|f\|_{L^p}\|g\|_{L^{p'}}.\]
    In particular, in the one weight case where $V = W^{-p'/p}$, if $W \in A_p$, we have
    \[\mc{A}_{\mc{S}}(W^{-1/p}f, W^{1/p}g) \lesssim_{d, \eta} [W]_{A_p}^{1 + \frac{1}{p-1} - \frac{1}{p}}\|f\|_{L^p}\|g\|_{L^{p'}}.\]
    As a corollary, for any operator $T$ that admits vector dual sparse domination, 
    \[|\ang{Tf, g}| \lesssim \mc{A}_{\mc{S}}(f, g),\]
    if $1 < p < \infty$ and $W \in A_p$, we have 
    \[ \|T\|_{L^p(W) \to L^p(W)} \lesssim_{p,d} [W]_{A_p}^{1 + \frac{1}{p-1} - \frac{1}{p}}.\]
\end{atheorem}
A direct consequence of \cref{MatrixWeightedThm} is the following.
\begin{corollary}\label{cor:martingale_multiplier}
    Let $\mu$ be a Radon measure on $\R^n$ and let $S \in \{T_\sigma, T\}$, where $T$ is a $L^1$ normalized shift as in \cref{D: L^1 normalized} of any complexity $(s,t)$, and $T_\sigma$ a martingale multiplier
    \[T_\sigma f = \sum_{Q \in \mc{D}}\sigma_Q \Delta_Q f\]
    where $\sigma = \{\sigma_Q\}_{Q \in \mc{D}}$ is a sequence with $\sigma_Q \in \{\pm 1\}$ and the martingale differences $\Delta_Q$ are defined by
    \[\Delta_Qf = \sum_{P \in \textup{ch}(Q)}\1_P(\ang{f}_P - \ang{f}_Q).\]
    Then if $1 < p < \infty$ and $W \in A_p$, we have
    \[ \|S\|_{L^p(W) \to L^p(W)} \lesssim_{p,d} [W]_{A_p}^{1 + \frac{1}{p-1} - \frac{1}{p}}.\]
\end{corollary}
We should mention that, in the scalar case and for $p=2$, weighted results for martingale multipliers were proved in \cite{ThieleTreilVolberg} and were stated for $L^1$ normalized shifts in \cite{Vasyunin_Volberg_2020} appealing to different techniques. \par 

The proof of \cref{MatrixWeightedThm} relies on a variable weight generalization of the dyadic Carleson Embedding Theorem which we regard as a meaningful contribution in its own right.

\begin{atheorem}\label{GeneralizedCarlesonEmbedding}
    Let $\{w_Q\}_{Q \in \mc{D}}$ be a collection of scalar weights indexed by $\mc{D}$ so that each $w_Q$ is localized to $Q$ (i.e. $w_Q \in L^1_{loc}(\mu)$ and $w_Q > 0$ $\mu$-almost everywhere on $Q$) and so that the following property is satisfied:
    \begin{equation}\label{eqn:compatibility}
        \|w_P w_{Q}^{-1}\|_{L^\infty(Q)} \leq A\ang{w_{P}}_{Q}\ang{w_{Q}}_{Q}^{-1} \quad \text{for all} \quad Q,P \in \mc{D}, \quad Q \subset P
    \end{equation}
    Let $\{\alpha_Q\}_{Q \in \mc{D}}$ be a sequence of real numbers indexed by $\mc{D}$. Then for $1 < p < \infty$,
    \begin{gather}
            \sum_{Q \in \mc{D}}\ang{w_Q^{1/p'}f}_Q^{p}\alpha_Q \leq C_1 \|f\|_{L^{p}}^{p} \quad \text{for all} \quad  f \in L^p(\mu)  \\ 
            \nonumber \text{if and only if}  \\ 
            \frac{1}{\mu(P)}\sum_{Q \in \mc{D}(P)}\ang{w_P}_{Q}\ang{w_Q}_Q^{p-1}\alpha_Q \leq C_2 \ang{w_P}_P \quad \text{for all} \quad P \in \mc{D} 
    \end{gather}
    Moreover, if $C_1$ and $C_2$ are the best constants in the above inequalities, then, 
    \[A^{-(p-1)}C_2 \leq C_1 \lesssim_{p} A^{1 + 1/p'}C_2.\]
\end{atheorem}
Note that by setting $w_Q = w$ for all $Q$, \cref{GeneralizedCarlesonEmbedding} recovers the usual (weighted formulation of the) Carleson Embedding Theorem. See \cref{section 4} for more details and a discussion about the compatibility condition \cref{eqn:compatibility}. \par
By combining Theorem \ref{thm: sparseforvectorhaar} with the proof of Theorem \ref{MatrixWeightedThm}, we provide quantitative weighted $L^p(W)$ estimates for vector Haar shifts of complexity $s+t=N$ when $W$ is a matrix weight belonging to a weight class $A_p^N$, which reflects the nonhomogeneous structure and properly contains $A_p$. These weighted estimates generalize the scalar results in \cite{CPW} to vector Haar shifts and matrix weights. We also show in Section \ref{ModifiedClass} that these weighted estimates are qualitatively sharp, at least in the case $n=1$ and for certain ``non-degenerate'' Haar systems when $n \geq 2.$ We postpone the precise definition of these weight classes to Section \ref{ModifiedClass}.

\begin{corollary} \label{cor:WeightedHaarShift}
Let $\mu$ a Radon measure in $\R^n$ and $\mathscr{H}$ a generalized Haar system such that the pair $(\mu, \mathscr{H})$ is balanced. Let $T$ be a vector Haar shift of complexity $(s,t)$ with $s+t=N$, $1<p<\infty$, and $W \in A_p^N$. Then there holds
$$ \|T  \|_{L^p(W) \to L^p(W)} \lesssim [W]_{A_p}^{1 + \frac{1}{p-1} - \frac{2}{p}}[W]_{A_p^N}^{\frac{1}{p}}$$
where the implicit constant depends only on $n, d, N, p, \mu,$ and $\mathscr{H}$ (not on $f$ or $W$). 
    
\end{corollary}

\subsection{Outline of the paper}  We discuss in more details the structure of the paper. In \cref{section 2} we introduce known results and we describe the balanced condition in higher dimensions, extending the discussions in \cite{LSMP} and \cite{CPW}. Differently from the one-dimensional case, the definition crucially involves both the underlying measure $\mu$ and the choice of the Haar system $\mathscr{H}$, so that we will refer to balanced pairs $(\mu, \mathscr{H})$ instead of balanced measures. In \cref{section 3} we prove \cref{thm: sparseforvectorhaar}, which is a combination of \cref{T:thm A 1st part} and \cref{T:thm A 2nd part}. The first deals with balanced pairs and general shifts, the second with general Radon measures and $L^1$ normalized shifts, which are defined in \cref{D: L^1 normalized}. In \cref{section 4} and \cref{ModifiedClass} we deal with matrix weighted $L^p$ estimates for the usual sparse form and the modified sparse forms arising in \cref{T:thm A 1st part}. Even if, as mentioned before, we do not use the Orlicz bumps technique as in \cite{cruzuribe2017} to deduce one weight results (the lack of reverse H\"{o}lder inequality being the major obstacle), we can still prove two weight inequalities for nonhomogeneous Haar shifts with Orlicz bump conditions on the weights; see Section \ref{ModifiedClass} for details.

\subsection{Acknowledgments}
We thank Sergei Treil for numerous helpful discussions related to convex body domination and matrix weighted inequalities. We also thank Jos\'{e} Conde Alonso, Jill Pipher and Brett Wick for their useful input on the structure and content of this paper.

\section{Preliminaries} \label{section 2}
\subsection{Haar systems and the balanced condition}

In this section, we introduce the balanced condition in $\R^n$. Differently from the one-dimensional case where, given a fixed dyadic grid $\D$ and a Borel measure $\mu$, there is a unique -up to sign- choice for a Haar system normalized in $L^{2}(\mu)$, in higher dimensions the balanced condition depends both on the measure and the choice of the Haar system. In what follows $\mu$ is a Borel measure on $\R^n$, $n\geq 1$ such that $0<\mu(Q)<\infty$ for every $Q\in \mathcal{D}$. We further assume that each quadrant has infinite measure.

\begin{definition}
We say $\mathscr{H}_\D=\{h_Q\}_{Q \in \D}$ is a generalized Haar system in $\R^n$ if the following holds:
\begin{enumerate}
    \item for every $Q \in \D$ we have $\mathrm{supp}(h_Q) \subset Q$;
    \item for every $R \in \mathcal{D}(Q)$, $R\subsetneq Q$, $h_Q$ is constant on $R$; in particular \[h_Q(x)=\sum_{R \in ch(Q)} \alpha_R \1_R(x),\] where  $\alpha_R \in \R$ and the set of $2^n$ children of $Q$ is denoted as \[ ch(Q)= \left\{ R \in \D: R \subsetneq Q \text{ and } \ell(R)=\frac{1}{2}\ell(Q) \right\}. \]  Here, $\ell(R)$ is the side-length of $R$; 
    \item for every $Q \in \D$, $h_Q$ has zero mean, i.e. $\int_Q h_Q(y)d\mu(y)=0$;
    \item for every $Q \in \D$, we have $\|h_Q\|_{L^2(\mu)}=1.$

\end{enumerate}
Furthermore, we say $\mathscr{H}$ is standard if 
\begin{equation} \label{D: standard}
    \Xi \left[\mathscr{H}, 0,0 \right]:= \sup_{Q \in \D} \|h_Q\|_{L^1(\mu)} \|h_Q\|_{L^\infty(\mu)} <\infty.   
   \end{equation}
\end{definition}

\begin{remark}
Using these conditions, one readily shows that the functions $\{h_Q\}$ form an orthonormal set in $L^2(\R^n).$ Notice that such a collection of functions $\mathcal{H}_{\D}$ may not form an orthonormal basis for $L^2(\R^n)$; i.e. the functions may not span all of $L^2(\R^n).$ 
The assumption $0<\mu(Q) < \infty$ is superfluous, in that if we allow $\mu$ to vanish on some dyadic cube in $\D$ and if $\mu(Q)=0$, we can set $h_Q\equiv0$; also, if $\mu(Q)=\mu(R)$ for some child $R$ of $Q$, then every brother of $R$ has null measure and again we set $h_Q \equiv 0$. Therefore, we can set $\D_{\mathscr{H}} \subset \D$ to be the set of dyadic cubes $Q \in \D$ with non-vanishing associated Haar function $h_Q$. Note that $\mathscr{H}_{\D_{\mathscr{H}}}=\{h_Q\}_{Q \in\D_{\mathscr{H}}}$ is an orthonormal system. In light of this, we assume $0< \mu(Q) < \infty$ for simplicity.  Finally, condition \eqref{D: standard} is equivalent to asking weak $(1,1)$-continuity of Haar multipliers, as it was observed in \cite{LSMP}. 

\end{remark}

\begin{definition}\label{def: vectorhaarshift}
A (vector) Haar shift $T$ of complexity $(s,t)$ acting (a priori) on vector-valued functions $f \in L^2(\R^n; \R^d)$ takes the form
\begin{gather} \label{D: dyadic shifts}Tf(x)= \sum_{Q \in \D} T_Qf(x):= \sum_{Q \in \D} \int_Q K_Q(x,y)f(y)d\mu(y),
\end{gather}
where
\begin{gather*}
K_Q(x,y)=\sum_{\substack{J \in \D_s(Q)\\ K \in \D_t(Q)}} c_{J,K}^Qh_J(y)h_K(x), \quad \text{ and }\quad \sup_{Q,J,K} | c_{J,K}^Q| \leq 1.\end{gather*} 
If, in addition, one has $\inf_{Q,J,K} | c_{J,K}^Q|>0$, then we say that $T$ is a non-degenerate (vector) Haar shift of complexity $(s,t)$.
\end{definition} 

\begin{remark}
The normalization $\sup_{Q,J,K} | c_{J,K}^Q| \leq 1$ is assumed for convenience, but it can also be altered as we later discuss. We denote the collection of all Haar shifts of complexity $(s,t)$ with respect to a standard Haar system $\mathscr{H}_{\D}$ by $HS(s,t, \mathscr{H}_{\D})$.
    
\end{remark}
In \cite{LSMP}, the authors gave necessary and sufficient conditions for $(\mu, \mathscr{H})$ to have weak $(1,1)$ boundedness of scalar-valued Haar shifts (i.e. $d=1$ in Definition \ref{def: vectorhaarshift}).
\begin{theorem}[\cite{LSMP}, Theorem 2.11] \label{T: LSMP}
Let $\mu$ and $\mathscr{H}$ as above, $n\geq 1$ and $d=1$. A non-degenerate Haar shift $T$ of complexity $(s,t)$
is weak $(1,1)$ if and only if 
\begin{equation}\label{Hstquantity}
\Xi \left[\mathscr{H}, s, t \right]:= \sup_{Q \in \D} \left\{ \|h_J\|_{L^\infty(\mu)}\|h_K\|_{L^1(\mu)}:  J \in \D_s(Q), K \in \D_t(Q)\right\} < \infty.
\end{equation}
Furthermore, we have that $\|T\|_{L^1(\mu) \to L^{1, \infty}(\mu)} \lesssim_\Xi 2^{(s+t)n}.$
\end{theorem}
\begin{example}
    Suppose $n=1$ and $0< \mu(I)< \infty$, for every dyadic interval $I \in \D$. The Haar function associated with $I$ is 
    $$h_I(x)=\sqrt{m(I)} \bigg(\frac{\1_{I_+}(x)}{\mu(I_+)} - \frac{\1_{I_-}(x)}{\mu(I_-)} \bigg), \quad m(I):= \frac{\mu(I_+)\mu(I_-)}{\mu(I)},$$ and we have $$
\|h_I\|_{L^1(\mu)}= 2\sqrt{m(I)}, \quad \|h_I\|_\infty \sim \frac{1}{\sqrt{m(I)}}.
$$
Theorem \ref{T: LSMP} says that every shift of complexity $(0,1)$ and $(1,0)$ is weak type $(1,1)$ if and only if 
$$m(I) \sim m(\widehat{I}), \quad \text{ for all } I \in \D,$$
i.e., $\mu$ is balanced.
\end{example}

We also have the following lemma.

\begin{lemma} If $\Xi \left[\mathscr{H}, 1,0 \right]< \infty$ and $ \Xi \left[\mathscr{H}, 0,1 \right] < \infty$, then every scalar-valued Haar shift defined with respect to $\mathscr{H}$ is weak $(1,1).$ 
\end{lemma}
\begin{proof}
We proceed by induction on the total complexity $N=s+t$ of the shift; if $N=0$, the statement follows from the fact that $\mathscr{H}$ is standard. Now assume that $s+t=N \geq 1$ and $\Xi \left[\mathscr{H}, m,n \right]< \infty$ if $m+n \leq N-1$. Either $t \geq 1$ or $s \geq 1$, so assume the first case. For any $J \in \D_s(Q)$ and $K \in \D_t(Q)$, we have that $\widehat{K} \in \D_{t-1}(Q)$ and using the inductive hypothesis 
\begin{align*}
    \|h_J\|_{L^\infty(\mu)} \|h_K\|_{L^1(\mu)} = &\|h_J\|_{L^\infty(\mu)} \|h_{\widehat{K}}\|_{L^2(\mu)}^2\|h_K\|_{L^1(\mu)}  \\ \leq & \|h_J\|_{L^\infty(\mu)} \|h_{\widehat{K}}\|_{L^1(\mu)}\|h_{\widehat{K}}\|_{L^\infty(\mu)}\|h_K\|_{L^1(\mu)} \\ \leq & \Xi \left[\mathscr{H}, s,t-1 \right] \Xi \left[\mathscr{H}, 0,1 \right] < \infty. 
\end{align*} 
In particular, for every $N$ and $s+t \leq N$, we get that $\Xi \left[\mathscr{H}, s,t \right] < \infty $. From Theorem \ref{T: LSMP}, that implies weak $(1,1)$ boundedness for Haar shifts of any complexity, since $N$ is arbitrary.
\end{proof}

Given a pair $(\mu, \mathscr{H})$, define the quantities 
\begin{equation}\label{def:m(Q)}
    m(Q)= m_{\mu, \mathscr{H}}(Q) := \|h_Q\|_{L^1(\mu)}^2.
\end{equation}

Up to a constant factor, this agrees with the definition of $m$ given in the $n=1$ case.   

\begin{definition} \label{def of balanced pairs}
    We say that a pair $(\mu, \mathscr{H})$ is balanced if $\mathscr{H}$ is standard, i.e. $\Xi \left[\mathscr{H}, 0,0 \right] <\infty$, and
    \begin{equation} \label{balanced_higher_dim}
    m(Q) \sim m(\widehat{Q}), \quad \text{ for every $Q \in \D$}
    \end{equation}
\end{definition}

 The following proposition collects some relevant properties of $m$ and of balanced pairs $(\mu, \mathscr{H})$.

\begin{proposition}\label{Prop: propertiesofm}
    Let $\mathscr{H}=\{h_{Q}\}_{Q\in \mathcal{D}}$ be a generalized Haar system in $\R^n$ equipped with a Borel measure $\mu$, and let $m(Q)$ as in (\ref{def:m(Q)}). Then
\begin{enumerate}
    \item $m(Q)\leq \mu(Q)$ for all $Q\in \mathcal{D}$.\label{part1prop}
    \item If $\mathscr{H}$ is standard, then $\|h_Q\|_{L^{\infty}(\mu)}\sim \frac{1}{\sqrt{m(Q)}}$, for all $Q\in \mathcal{D}$.\label{part2prop}
    \item If $(\mu,\mathscr{H})$ is balanced, then one has that $$m(Q)\sim \min\{\mu(R)\colon R\in \textup{ch}(Q)\},$$
    with implicitly constant depending on $n,\Xi[\mathscr{H},1, 0] \text{ and }\Xi[\mathscr{H},0, 1]$.\label{part3prop}
    \item $(\mu,\mathscr{H})$ is balanced if and only if \[\Xi \left[\mathscr{H}, 1,0 \right]< \infty \quad \text{ and } \quad \Xi \left[\mathscr{H}, 0,1 \right] < \infty.\] \label{part4prop}
    
\end{enumerate}
\end{proposition}

\begin{remark} \label{rmk on m}
Observe that claim (\ref{part3prop}) in the proposition tells us that, given two Haar systems $\mathscr{H}$ and $\widetilde{\mathscr{H}}$ such that $(\mu, \mathscr{H})$ and $(\mu, \widetilde{\mathscr{H}})$ are balanced pairs, we have that 
\[ m_{\mu, \mathscr{H}}(Q) \sim m_{\mu, \widetilde{\mathscr{H}}}(Q), \quad \text{ for every } Q \in \D. \]
\end{remark}

\begin{proof}

 Claim (\ref{part1prop}) follows easily from the definition of $m$ and the $L^2$ normalization,
    \begin{equation}\label{E: misboundedbymeasure} \sqrt{m(Q)} = \|h_Q\|_{L^1(\mu)} \leq \sqrt{\mu(Q)}\|h_Q\|_{L^2(\mu)} = \sqrt{\mu(Q)}.\end{equation}

    Claim (\ref{part2prop}) is also immediate from the chain of inequalities:
    $$1=\|h_Q\|_{L^2(\mu)}^{2}\leq \|h_Q\|_{L^{1}(\mu)}\|h_{Q}\|_{L^{\infty}(\mu)}=\sqrt{m(Q)}\|h_{Q}\|_{L^{\infty}(\mu)}\leq \Xi[\mathscr{H},0,0].$$

    To prove (\ref{part3prop}), we start by observing that from the definitions of $m(Q)$ and standard Haar system we know that for any $Q\in \mathcal{D}$,
    \begin{equation}\label{E: coefficientcontrolbym(Q)}
     m(Q)^{1/2}=\|h_Q\|_{L^1(\mu)}=\sum_{R\in \text{ch}(Q)}|\alpha_R| \mu(R)
    \quad \Rightarrow \quad |\alpha_R| \mu(R)\leq \sqrt{m(Q)}\text{ for all }R\in \text{ch}(Q).
    \end{equation}

     From that we have
     \begin{align*}
    1=\|h_Q\|_{L^2(\mu)}^2=\sum_{R\in \text{ch}(Q)}|\alpha_R|^2 \mu(R)
    \leq & \,m(Q)\,\sum_{R\in \text{ch}(Q)} \frac{1}{\mu(R)}\\
    \leq & 2^n m(Q)\left(\min\{\mu(R)\colon R\in \text{ch}(Q)\}\right)^{-1},
    \end{align*}
    which guarantees $\min \{\mu(R)\colon R\in \text{ch}(Q)\}\lesssim_n m(Q).$ For the other inequality, we use the balanced condition and (\ref{E: misboundedbymeasure}) to get
    $$m(Q)\lesssim m(R)\leq \mu(R)\text{ for all }R\in \text{ch}(Q). $$
    
     Next, we prove (\ref{part4prop}). If $(\mu,\mathscr{H})$ is balanced the for any cube $Q$

     \[\|h_{\widehat{Q}}\|_{L^\infty(\mu)}\|h_Q\|_{L^1(\mu)} \leq \Xi[\mathscr{H},0, 0]\frac{\|h_Q\|_{L^1(\mu)}}{\|h_{\widehat{Q}}\|_{L^1(\mu)}} = \Xi[\mathscr{H},0, 0]\sqrt{\frac{m(Q)}{m(\widehat{Q})}} \lesssim 1.\]
    This proves that $\Xi[\mathscr{H},0,1] < \infty$. Similarly, we can show that $\Xi[\mathscr{H},1,0] < \infty$.

    Conversely, suppose that $\Xi \left[\mathscr{H}, 1,0 \right]< \infty$ and $\Xi \left[\mathscr{H}, 0,1 \right] < \infty$. Then $\mathscr{H}$ is standard as 

    \[\|h_Q\|_{L^\infty(\mu)}\|h_Q\|_{L^1(\mu)} 
    \leq \Xi[\mathscr{H},1, 0]\Xi[\mathscr{H},0, 1].\]

    Also, we have that
    \[1 = \|h_Q\|_{L^2}^2 \leq \|h_Q\|_{L^\infty(\mu)}\|h_Q\|_{L^1(\mu)} \leq \Xi[\mathscr{H},1, 0] \frac{\|h_Q\|_{L^1(\mu)}}{\|h_{\widehat{Q}}\|_{L^1(\mu)}} = \Xi[\mathscr{H}, 1, 0] \frac{\sqrt{m(Q)}}{\sqrt{m(\widehat{Q})}}.\]
    
    Arguing similarly with $\Xi[\mathscr{H},0, 1]$, we have 
    \[\Xi[\mathscr{H}, 0, 1]^{-1} \leq \sqrt{\frac{m(\widehat{Q})}{m(Q)}} \leq \Xi[\mathscr{H},1, 0].\]
    Hence, $(\mu, \mathscr{H})$ is balanced and this concludes the proof of (\ref{part4prop}).

\end{proof}
\begin{remark} We make some comments on the previous statements.
 \begin{enumerate}
     \item  We defined a balanced pair $(\mu, \mathscr{H})$ as a standard Haar system $\mathscr{H}$ such that $m(Q) \sim m(\widehat{Q})$ for any $Q \in \D$. In light of \cref{Prop: propertiesofm}, if this holds then any scalar shift defined with respect to $\mathscr{H}$ is weak $(1,1)$ with respect to $\mu$.
    This is coherent with the one-dimensional case or with the simpler case of two-valued Haar systems in $\R^n$ (see \cite{LSMP}, Remark 2.13), for which the statement $\Xi[\mathscr{H}, 0, 0] < \infty$ in the definition of balanced is superfluous, but also allows to consider more general systems like wavelet Haar systems. 
    \item In higher dimensions the balanced condition crucially depends both on the measure and on the choice of the Haar system. Indeed, as proved in \cite{LSMP}, given a product measure $d \mu(x,y)= dx d\nu(y)$ such that $\nu$ is balanced but not dyadically doubling, one can build two Haar systems $\mathscr{H}$ and $\widetilde{\mathscr{H}}$ such that $(\mu, \mathscr{H})$ is balanced but $(\mu, \widetilde{\mathscr{H}})$ is not.
 \end{enumerate}   
\end{remark}

We conclude with a brief discussion on a subclass of Haar shifts. Suppose $T$ is a shift of complexity $(s,t)$ as in \eqref{D: dyadic shifts}. If $\mu$ is doubling, it is easy to see that the kernels $K_Q$ of $T_Q$ verify
$$\|K_Q\|_{\infty} \lesssim_{\mu,s,t} \frac{1}{\mu(Q)}, \quad \text{ for any $Q \in \D $.}$$
where the implicit constant depends exponentially on $s$ and $t$. This is not true in general if the measure is non-doubling. Indeed, even if the measure is balanced, all we can say is 
$$\|K_Q\|_{\infty} \lesssim_{\mu,s,t} \frac{1}{m(Q)}, \quad \text{ for any $Q \in \D $,}$$
and one can have $m(Q) \lll \mu(Q)$.

\begin{definition}\label{D: L^1 normalized}
Given a Radon measure $\mu$ on $\R^n$ such that $0< \mu(Q) < \infty$, we say a Haar shift $Tf(x)= \sum_{Q \in \D}T_Q f(x)$ defined as in \eqref{D: dyadic shifts} is $L^1$ normalized if \begin{equation} \label{E: L1 norm}\|K_Q\|_{\infty}  \lesssim_{\mu} \frac{1}{\mu(Q)}, \quad \text{ for any $Q \in \D$}.\end{equation}
\end{definition}
\begin{remark}
The normalization \eqref{E: L1 norm}, where the implicit constant does not depend on the complexity, usually holds for Haar shifts arising from representation theorems of Calder\'on-Zygmund operators in the doubling setting (see \cite{HPTV}).  The worse behavior of the kernels $K_Q$ seems to essentially be the reason of the failure of usual sparse domination for general Haar shifts, as we will highlight in the next sections.
\end{remark}
\subsection{Sparse families and sparse operators} 
We recall some basic facts about sparse families. 

\begin{definition}
    Let $\mc{S} \subset \mc{D}$ be a family of dyadic cubes.
    \begin{enumerate}
        \item Let $0 < \eta < 1$. We say that $\mc{S}$ is $\eta$-sparse if for each $Q \in \mc{S}$, there exists some Borel set $E_Q \subset Q$ so that $\mu(E_Q) \geq \eta \, \mu(Q)$ and the collection $\{E_Q\}_{Q \in \mc{S}}$ is pairwise disjoint. 
        \item Let $\Lambda > 0$. We say that $\mc{S}$ is $\Lambda$-Carleson if for every sub-collection $\mc{S}' \subset \mc{S}$, we have 
        \[\sum_{Q \in S'}\mu(Q) \leq \Lambda \,\mu\left(\bigcup_{Q \in \mc{S}'}Q\right)\]
    \end{enumerate}
\end{definition}

It was shown in \cite{Hanninen} that if the measure $\mu$ has no point masses, then $\mc{S}$ is $\eta$-sparse if and only if $\mc{S}$ is $\eta^{-1}$-Carleson. \par
Crucial to our study of dyadic operators is the convex set-valued sparse operator
\[\mc{L}_{\mc{S}}f(x) = \sum_{Q \in \mc{S}}\dang{f}_Q\1_{Q}(x)\]
and the associated sparse form
\[\mc{A}_{\mc{S}}(f, g) := \sum_{Q \in \mc{S}}\dashint_{Q}\dashint_{Q}|f(x) \cdot g(y)| \, d\mu(x) \, d\mu(y)\mu(Q)\]
defined for vector-valued functions $f$ and $g$. \par
The following lemma tells us that the pointwise operator can be controlled by the sparse form. For a symmetric, closed, and convex set $\mc{C}$, and $v \in \R^d$, we write
\[\mc{C} \cdot v = \{x \cdot v : x \in \mc{C}\}\]
which is a symmetric closed interval in $\R$. We write $|\mc{C} \cdot v|$ for the right endpoint of this interval. 

\begin{lemma}\label{lem:pointwise_to_form}
    Let $f$ and $g$ be vector-valued locally integrable functions. Then 
    \[\int_{\R^n}|\mc{L}_{\mc{S}}f(y) \cdot g(y)| \, d\mu(y) \leq \mc{A}_{\mc{S}}(f, g).\]
\end{lemma}
\begin{proof}
    Observe that 
    \[\dang{f}_{Q} \cdot g(y) = \{\ang{\varphi f \cdot g(y)}_Q : \|\varphi\|_{L^\infty} \leq 1\}.\]
    This is a symmetric interval in $\R$. For any element in this interval, we have 
    \begin{align*}
        |\ang{\varphi f \cdot g(y)}_Q| & = \left|\dashint_Q \varphi(x)f(x) \cdot g(y) \, d\mu(x)\right| \\
        & \leq \dashint_Q |f(x) \cdot g(y)| \, d\mu(x).
    \end{align*}
    It follows that 
    \begin{align*}
        |\mc{L}_{\mc{S}}f(y) \cdot g(y)| & \leq \sum_{Q \in \mc{S}}|\dang{f}_Q \cdot g(y)|\1_Q(y) \\
        & \leq \sum_{Q \in \mc{S}}\dashint_Q |f(x) \cdot g(y)| \, d\mu(x)\1_Q(y)
    \end{align*}
    By integrating this inequality in the variable $y$, we conclude the proof of the lemma.
\end{proof}

As an immediate consequence of these manipulations, if $T$ is an operator that admits the convex body domination
\[Tf(x) \in C\mc{L}_{\mc{S}}f(x)\]
(where $\mc{S}$ is a sparse family that could depend on $f$, but whose sparse parameter is independent of $f$), we obtain
\[|\ang{Tf, g}| \lesssim \mc{A}_{\mc{S}}(f, g).\]
By this observation, weighted estimates for operators $T$ that admit vector sparse domination follow from weighted estimates for the sparse form $\mc{A}_{\mc{S}}$ as in \cref{MatrixWeightedThm}. See \cref{section 4} for more details. 

\subsection{Matrix weights}
We introduce some known facts about matrix weights, and some results about modified weigh classes. Let $\mathcal{M}_{d \times d}^{+}$ denote the space of $d \times d$ positive semi-definite matrices with real entries. A matrix weight $W$ is a locally integrable function $W: \R^n \rightarrow \mathcal{M}_{d \times d}^{+}$. 
Given a matrix weight $W$ and $x \in \R^n$, we define the operator norm of $W(x)$ by \[\bm{|}W(x) \bm{|}= \sup_{\substack{e \in \R^d \\ |e|=1}} |W(x)e| \]
The Lebesgue space $L^p(W)$ for $1 \leq p<\infty$ is given by the set of measurable, vector-valued functions $f: \R^n \rightarrow \R^d$ whose following norm is finite:

$$\|f\|_{L^p(W)}:= \left( \int_{\R^n} |W^{1/p}(x)f(x)|^p \, d \mu(x) \right)^{1/p}.$$
\begin{definition}
For $1 < p < \infty$, and $W,V$ matrix weights, the matrix $A_p$ constant is
\[[W, V]_{A_p} = \sup_{Q \in \D}\dashint_{Q}\left(\dashint_{Q}\bm{|}W(x)^{1/p}V(y)^{1/p'}\bm{|}^{p'} \, d\mu(y)\right)^{p / p'}d\mu(x).\]
In the single weight case, we have $V= W^{-p' / p}$ so that 
\[[W]_{A_p} := \sup_{Q \in \D}\dashint_{Q}\left(\dashint_{Q}\bm{|}W(x)^{1/p}W(y)^{-1/p}\bm{|}^{p'} \, d\mu(y)\right)^{p / p'}d\mu(x).\]
\end{definition}
\begin{definition}
For a matrix weight $W$, the scalar $A_{p, \infty}$ constant is
\[[W]_{A_{p, \infty}^{sc}} := \sup_{e  \in \R^d}[|W^{1/p}e|^p]_{A_\infty},\] where for a scalar weight $w$, the $A_\infty$ constant is defined as 
\[[w]_{A_\infty} := \sup_{Q \in \D}\frac{1}{w(Q)}\int_{Q}M^\D_Q(w \chi_Q)(x) \, d\mu(x)\]
and $M^\D_Qf(x) := \sup_{x \in Q' \in \mc{D}(Q)}\ang{|f|}_{Q'}$ is the dyadic maximal function.
\end{definition}
 We collect some known facts about matrix weights, that will be used in \cref{section 4} and \cref{ModifiedClass}. We refer to \cite{Goldberg2003} and \cite{cruzuribe2017} for more details. We remark that although these results are stated for the special case when $\mu$ is the Lebesgue measure on $\R^d$, their proofs apply verbatim to any Borel measure $\mu$ as considered in our setting.

\begin{proposition}\label{prop:weight_duality} For any $1< p <\infty$ and for every matrix weight $W \in A_p$ and , we have that
 \[[W]_{A_{p, \infty}^{sc}} \leq[W]_{A_p}.\]
 Also, we have that $W^{-\frac{p'}{p}} \in A_{p'}$ and 
 \begin{equation}\label{p-p' relation}
 [W]_{A_p}^{\frac{1}{p}} \sim [W^{-\frac{p'}{p}}]_{A_{p'}}^{\frac{1}{p'}}.    
 \end{equation}
\end{proposition}
Now we introduce the reducing operators and related facts.
\begin{proposition} \label{reducing operators}
For any $1<p<\infty$, any $Q \in  \D$ and any matrix weight $W$, there exist $\mathcal{W}_{p, Q} \in \mathcal{M}_{d \times d}^{+}$, called a reducing operator, satisfying for each vector $e \in \R^d$:
$$ |\mathcal{W}_{p,Q} e |^p \sim_d \dashint_{Q} |W^{1/p}(x)e |^p \, d\mu(x).$$   
Moreover, given two matrix weights $W,V$, if $\mc{V}_{p',Q}$ denotes the reducing operator 
$$ |\mathcal{V}_{p',Q} e |^{p'} \sim_d \dashint_{Q} |V^{1/p'}(x)e |^{p'} \, d\mu(x).$$  
we have that 
\begin{equation} \label{prop reducing op}
\bv{\mc{W}_{p,Q}\mc{V}_{p',Q}} \lesssim_d [W, V]_{A_p}^{1/p}.   
\end{equation}
\end{proposition}

\begin{proposition} \cite[Proposition 2.1]{Goldberg2003} \label{ExpOp}
Let $\mathbb{E}_Q$ denote the vector-valued expectation operator on the dyadic cube $Q$; i.e.
$$ \mathbb{E}_Q f:= \langle f \rangle_Q \mathbf{1}_Q.$$
There holds 
$$\sup_{Q \in \D} \|\mathbb{E}_Q\|_{L^p(W) \rightarrow L^p(W)}^p \sim [W]_{A_p}.$$ 
\end{proposition}

To study weighted estimates for Haar shifts, we must introduce weight classes that take into account the complexity of the shift. As always we are assuming that the underlying measure $\mu$ is finite and nonzero on every dyadic cube, infinite on quadrants, and atomless, and that the pair $(\mu, \mathscr{H})$ is balanced. 
\begin{definition}
    Let $1<p<\infty$ and $N\in\N$. Given cubes $Q,R\in\mathcal D$, we denote
    \begin{equation*}
        c_p^b(Q,R)=\begin{cases}
            1,\textrm{ if }Q=R,\\
            \frac{m(Q)^{p/2}m(R)^{p/2}}{\mu(R)\mu(Q)^{p-1}},\textrm{ otherwise}.
        \end{cases}
    \end{equation*}
    Given a weight $W$, we say $W\in A_p^N$ if
    \begin{equation*}
        [W]_{A_p^N}:=\sup_{\substack{Q,R\in\mathcal D\\0\leq\mathrm{dist}(Q,R)\leq N+2}}c_p^b(Q,R)\dashint_Q\left(\dashint_R\bm| W(x)^{1/p}W(y)^{-1/p}\bm|^{p'}d\mu(y)\right)^{p/p'} d\mu(x)<\infty.
    \end{equation*}
\end{definition}
\begin{remark}
The definition of $[W]_{A_p^N}$ involves $m(Q)$  and so, \textit{a priori}, it depends on the choice of Haar system, in addition to the weight $W$ and the measure $\mu$. However, given two balanced pairs $(\mu, \mathscr{H})$ and $(\mu, \widetilde{\mathscr{H}})$, we have seen in \cref{rmk on m} that 
\[ m_{\mu, \mathscr{H}}(Q) \sim m_{\mu, \widetilde{\mathscr{H}}}(Q), \quad \text{ for every } Q \in \D. \]
In particular, weighted estimates depending on the modified weight classes are equivalent for Haar shifts defined with respect to different Haar systems, as long as they satisfy the balanced condition. 
\end{remark}

\section{Pointwise convex body domination for vector Haar shifts} \label{section 3}

In this section, we will prove a modified convex body domination principle for vector Haar shifts of complexity $(s,t)$. Let $\mathcal{H}_{\mathcal{D}}=\{h_{Q}\}_{Q\in \mathcal{D}}$ be a generalized Haar system in $\R^n$. 

To simplify notation we consider shifts of the form \[Tf(x) = \sum_{Q \in \D} T_Qf(x)= \sum_{Q \in \mc{D}}c_Q\ang{f, h_{Q_1}}h_{Q_2}(x)\]
    where $Q_1$ denote some interval in $\mc{D}_s(Q)$, with fixed location among the $2^{sn}$ possibilities. Similar with $Q_2\in \mc{D}_t(Q)$. There is no loss of generality in considering these Haar shifts since any such operator as in Definition \ref{def: vectorhaarshift} can be written as a sum of finitely many of those elementary shifts.

\begin{definition} \label{D: t-separated}
    A Haar shift is $t$-separated if there exists $k \in \{0, \dots, t\}$ such that $T_Q \neq 0$ only if $Q \in \sq{\mc{D}}^k := \bigcup_j \D^{k+j(t+1)}$.
\end{definition}
 As $\D= \bigcup_{k \in \{ 0, \dots, t\}} \sq{\mc{D}}^k$, every Haar shift is a finite sum of at most $(t+1)$ $t$-separated ones. 
\begin{remark} \label{R: properties of t-sep shift}The purpose of this reduction is to guarantee that for any cube $J \in \sq{\mc{D}}$,
    \[\sum_{Q \in \mc{D}: Q \supsetneq J}c_{Q} \ang{f, h_{Q_1}}h_{Q_2}\]
    is constant on $J$.
\end{remark}
Recall that, for any compact convex set $K \subset \R^d$ with non-empty interior, there exists an ellispoid $\mc{E}$, called the \emph{John ellispoid} of $K$ such that 
\[\mc{E} \subset K \subset \sqrt{d} \mc{E}.\] 
\begin{lemma} [\cite{NPTV}] Let $f: \R^n \to \R^d$ be integrable in $Q \in \D$ and non-trivial, i.e. $f(x) \neq 0$ on a set of positive measure. There exists a unique subspace $E \subset \R^d$ containing $\dang{f}_Q$ such that $\dang{f}_Q$ has non-empty interior in $E$.
In particular $E= \text{Ran} (F_f)$, where \[F_f: L^\infty(Q) \to \R^d, \quad F_f(\psi)= \ang{f \psi}_Q. \]
\end{lemma}
\begin{definition}
Given a non-trivial $f \in L^1(Q)$ as before, the John ellipsoid of $\dang{f}_Q$ is defined as the John ellipsoid in the subspace $E$.        
\end{definition} 

Notice that $\mc{E}$ need not be non-degenerate, i.e. it might not have full dimension in $\R^d$. The following lemma will be an important tool in transferring information from scalar Haar shifts to vector-valued ones. 

\begin{lemma}\label{lem:convex_body_components} Let $f\in L^{1}(\R^n;\R^d)$ and let $\{e_j\}_{j = 1}^{d}$ be an orthonormal basis of $\R^d$ associated to the John ellipsoid $\mc{E}$ of $\dang{f}_Q$ in the sense that 
 \[\mc{E} = \left\{\sum_{j = 1}^{d}x_j\alpha_je_j : \sum_{j = 1}^{d}x_j^2 \leq 1\right\},\]
 for some $\alpha_j\in [0,\infty)$. If $v \in \R^d$ is a vector so that 
    \[|v \cdot e_j| \leq A\ang{|f \cdot e_j|}_Q\]
    for all $j$, then $v \in Ad \dang{f}_Q$.
\end{lemma}

\begin{proof}
  
Recall that $\mathcal{E}\subset \dang{f}_Q\subset  \sqrt{d}\mathcal{E} $ by the definition of John ellipsoid. By taking $\varphi_j = \text{sgn} (f \cdot e_j)$,
    \[\ang{|f \cdot e_j|}_Q e_j = \ang{\varphi_j (f \cdot e_j)e_j}_Q\]
    is a component of $\langle \varphi_j f\rangle_{Q}$, a vector in $\dang{f}_Q\subset  \sqrt{d}\mathcal{E}$.
    Therefore,  $\ang{|f \cdot e_j|}_Q \leq \sqrt{d}\alpha_j$ for each $j$.
    
    It follows from the assumption that 
    \[|v \cdot e_j| \leq A\ang{|f \cdot e_j|}_Q \leq A \sqrt{d}\alpha_j.\]
    
    Since $\sum_{j=1}^{d}\left(\frac{|v \cdot e_j|}{Ad\alpha_j}\right)^2\leq 1$, this implies that
    \begin{align*}
        v = \sum_{j=1}^d (v \cdot e_j) e_j = Ad\sum_{j=1}^d \left(\frac{v \cdot e_j}{Ad\alpha_j}\right)\alpha_j e_j \in Ad\mc{E} \subset Ad\dang{f}_Q.
    \end{align*}
\end{proof}

Next we use Lemma \ref{lem:convex_body_components} to deduce a vector-valued variant of the weak $(1,1)$ inequality satisfied by scalar-valued Haar shifts $T$ in the context of balanced measures. In what follows, we refer to $\tilde{T}$ as a truncation of $T$ if $\tilde{T}$ is obtained by discarding some of the terms of the sum in $T$. 

\begin{lemma}\label{lem: vectorweak11}
Assume that $(\mu,\mathscr{H})$ is balanced. Let $T$ be a Haar shift of complexity $(s,t)$ and $f\in L^{1}(\R^n;\R^d)$ be compactly supported in $Q_0\in \mathcal{D}$. Then if $\sq{T}$ is any truncation of $T$ and $\lambda>0$ 
  \begin{align*}
        \mu (\{x \in Q_0: \sq{T}f(x) \not \in \lambda \dang{f}_{Q_0}\}) \leq \|T\|_{L^1 \to L^{1, \infty}}(d^2 / \lambda) \mu(Q_0);
    \end{align*}
where the norm $\|T\|_{L^{1}\rightarrow L^{1,\infty}}$ refers to the action of $T$ in scalar functions $f\in L^{1}(\R^n;\R)$.
\end{lemma}

\begin{proof}
For any truncation $\tilde{T}$ of $T$ one has $\|\sq{T}\|_{L^{1}\rightarrow L^{1,\infty}}\leq \|T\|_{L^1\rightarrow L^{1,\infty}}<\infty$, since $(\mu,\mathscr{H})$ is balanced. Using that combined with Lemma \ref{lem:convex_body_components} one has that 
 \begin{align*}
        \mu(\{x \in Q_0: \sq{T}f(x) \not \in \lambda \dang{f}_{Q_0}\}) & \leq \sum_{j = 1}^{d}\mu(\{x \in Q_0 : |\sq{T}(f \cdot e_j)(x)| > (\lambda / d)\ang{|f \cdot e_j|}_{Q_0}\}) \\
        & \leq \sum_{j = 1}^{d}\|T\|_{L^1 \to L^{1, \infty}}\frac{d}{\lambda} \frac{\int_{Q_0}|f \cdot e_j| \, d\mu}{\ang{|f \cdot e_j|}_{Q_0}} \\
        & =\frac{d^2}{ \lambda} \|T\|_{L^1 \to L^{1, \infty}}\mu(Q_0).
    \end{align*}    \end{proof}

\begin{remark}
   The previous lemma holds in fact for any scalar linear operator $T$ which is weak $(1,1)$ with respect to a Radon measure $\mu$, when applied to a vector valued function $f$ as before.
\end{remark}
 Given a sparse family $\mathcal{S} \subset\mathcal{D}$, $N\in \N$ and two functions $f,g\in L^{1}(\R^n;\R^d)$, consider the bilinear sparse forms
\begin{equation*}
    \begin{split}
        \mathcal{A}_{\mathcal{S}}(f,g)=&\sum_{Q \in \mc{S}}\dashint_{Q}\dashint_{Q}|f(x) \cdot g(y)| d\mu(x) \, d\mu(y)\mu(Q),\\
    \mathcal{A}_{\mathcal{S}}^{N}(f,g)=&\sum_{\substack{ J, K \in \mc{S} \\ \text{dist}(J, K) \leq N+2}}\dashint_J\dashint_K |f(x)\cdot g(y)|d\mu(x)d\mu(y)\sqrt{m(J)}\sqrt{m(K)},
    \end{split}
    \end{equation*}
for $m(J)$ defined as in (\ref{def:m(Q)}).

When $d=1$, such sparse bilinear forms recover the ones introduced in \cite{CPW}. However, when $d\geq 2$, these bilinear sparse forms carry much more information of $f$ and $g$ than what one would have by naively taking products of averages of $|f|$ and $|g|$ in sparse cubes, since not only the amplitude of $|f(x)|$ and $|g(y)|$ come into play, but also their directions. 

Now we prove the first part of Theorem \ref{thm: sparseforvectorhaar}.
\begin{theorem} \label{T:thm A 1st part}
Let $\mu$ a Radon measure in $\R^n$ and $\mathscr{H}$ a generalized Haar system such that the pair $(\mu, \mathscr{H})$ is balanced. Let $f \in L^1(\R^n;\R^d)$ be compactly supported in $Q_0 \in \D$, and $T$ be a vector valued Haar shift of complexity $(s,t)$. There exists a sparse family $\mathcal{S}= \mathcal{S}(f)\subset \mathcal{D}(Q_0)$ and a constant $C>0$ such that  
\[Tf(x) \in C\sum_{Q \in \mc{S}}\dang{f}_Q \1_{Q}(x) + C\sum_{\substack{ J, K \in \mc{S} \\ \text{dist}(J, K) \leq s+t + 2 }}\dang{f}_{J} \frac{\1_K(x)}{\mu(K)}\sqrt{m(J)}\sqrt{m(K)} \quad \text{ on } Q_0.\] 
\end{theorem}
Before proving the theorem, we state an immediate corollary that follows from \cref{lem:pointwise_to_form}.
\begin{corollary}
    Under the hypothesis of Theorem \ref{thm: sparseforvectorhaar}, for every $f,g\in L^{1}(\R^n;\R^d)$ supported in $Q_0$, there exists a sparse family $\mathcal{S}$ such that 
    \begin{equation}
        \left|\int_{Q_0} Tf(x)\cdot g(x) d\mu(x)\right|\lesssim \mathcal{A}_{\mathcal{S}}(f,g)+\mathcal{A}_{\mathcal{S}}^{N}(f,g). 
    \end{equation}    
\end{corollary}

\begin{proof}[Proof of \cref{T:thm A 1st part}]
    As mentioned before, we can assume that $T$ can be written as 
    \[Tf(x) = \sum_{Q \in \D} T_Qf(x)= \sum_{Q \in \mc{D}}c_Q\ang{f, h_{Q_1}}h_{Q_2}(x), \quad \sup_{Q\in \mathcal{D}}|c_Q|\leq 1 \]
    where $Q_1 \in \mc{D}_s(Q)$ and $Q_2\in \mc{D}_t(Q)$ have fixed positions, and let $T$ be $t$-separated as in Definiton \ref{D: t-separated}. We have 
    \[c_Q \neq 0 \Rightarrow Q \in \sq{\mc{D},}\]
    where $\sq{\mc{D}} = \sq{\mc{D}}^k$ for some $k$. We can assume $Q_0 \in \sq{\mc{D}}$, otherwise we can replace $Q_0$ with a larger cube. For a cube $J\in \mathcal{D}$, let $T^J$ denote the local operator 
    \[T^Jf(x) = \sum_{Q \in \mc{D}(J)}c_Q\ang{f, h_{Q_1}}h_{Q_2} = \sum_{Q \in \mc{D}(J)}T_Qf(x).\]
    Note that for $x\in Q_0$,
    \[Tf(x) = T^{Q_0}f(x) + \sum_{Q \in \mc{D}: Q \supsetneq Q_{0}}T_Qf(x)=:T^{Q_0}f(x) +\tilde{T}f(x).\]
    
    From Lemma \ref{lem: vectorweak11}, for any $C\geq C_0=2d^2\|T\|_{L^1 \to L^{1, \infty}}$ 
    \begin{align*}
        \mu(\{x \in Q_0: \sq{T}f(x) \not \in C\dang{f}_{Q_0}\}) \leq \|T\|_{L^1 \to L^{1, \infty}}(d^2 / C) \mu(Q_0) \leq \frac{1}{2}\mu(Q_0).
    \end{align*}

     Hence we have $\sq{T}f(x) \in C \dang{f}_{Q_0}$ for all $x\in Q_0$ since $\sq{T}f$ is constant on $Q_0$, according to Remark \ref{R: properties of t-sep shift}. Therefore 
    \[Tf(x) \in T^{Q_0}f(x) + C \dang{f}_{Q_0}\text{ for all }x\in Q_0, \]
    and so it suffices to bound the local operator $T^{Q_0}$. \\

    We define the stopping times as follows: for a cube $Q' \in {\mc{D}}(Q_0)$ , let $\mc{B}(Q')$ be the maximal cubes $J \in \mc{D}(Q')$ satisfying at least one of the following conditions:
    \begin{gather}
     J \in \sq{\mc{D}} \quad \text{and} \quad \sum_{Q \in \mc{D}(Q') : J \subsetneq Q}T_Q(f\1_{Q'})(x) \not \in C \dang{f}_{Q'} \,\, \text{on} \,\, J \label{E: stop1} \\
     \ang{f}_J \not \in C\dang{f}_{Q'}. \label{E: stop2}
    \end{gather}
    Let $\mc{Q}_1(Q') \subset \mc{B}(Q')$ denote the collection of cubes satisfying \eqref{E: stop1}. Consider the operator
    \[T^1 = \sum_{Q \in \mc{D}(Q') \setminus \bigcup_{J \in \mc{Q}_1(Q')}\mc{D}(J)}T_Q.\]
    Then if $x \in J \in \mc{Q}_1(Q')$, we have $T^1(f\1_{Q'})(x) \not \in C \dang{f}_{Q'}$. Therefore, using Lemma \ref{lem: vectorweak11} for $T^1$ in place of $\tilde{T}$ and choosing $C$ analogously as before
    \[\sum_{J \in \mc{Q}_1(Q')}\mu(J) \leq \mu\left(\{x \in Q' : T^1(f\1_{Q'})(x) \not \in C \dang{f}_{Q'}\}\right) \leq \frac{1}{4}\mu(Q').\]
    
    Similarly, we can use Lemma \ref{lem:convex_body_components} and the weak $(1,1)$ bound for the dyadic Hardy-Littlewood maximal function to bound the sum of the measures of the cubes satisfying \eqref{E: stop2} by $\frac{1}{4}\mu(Q')$ as long as $C\gg1$. All together we get
    \[\sum_{J \in \mc{B}(Q')}\mu(J) \leq \frac{1}{2}\mu(Q')\]
    for all $Q' \in \mc{D}(Q_0)$. We now form a sparse family $\mc{S}'$ in the standard way: set $\mc{B}_0(Q_0) := \{Q_0\}$ and inductively define 
    \[\mc{B}_k(Q_0) := \bigcup_{Q \in \mc{B}_{k-1}(Q_0)}\mc{B}(Q).\]
    
    The family
    \[\mc{S}' = \bigcup_{k=0}^{\infty} \mc{B}_k(Q_0)\]
    is then $\frac{1}{2}$-sparse. For technical reasons that will become clear later, it will be convenient to enlarge such a sparse family by adding all descendants of order at most $t$ of any cube originally in $\mathcal{S}'$, and let 
\begin{equation}\label{enlargedsparse}
    \mc{S} := \bigcup_{Q\in \mathcal{S}'}\mathcal{D}^{t}(Q),
\end{equation}
where $\mathcal{D}^{t}(I):=\bigcup_{l=0}^{t}\mathcal{D}_l(I).$
The following lemma guarantees that $\mathcal{S}$ is $\eta$-sparse for some $\eta=\eta(t,d)\in (0,1)$.

    \begin{lemma} Let $\mu$ be a Borel measure in $\R^n$ with no atoms, and $\eta_1\in (0,1)$. If $\mathcal{S}'$
 is an $\eta_1$-sparse family of cubes in $\R^n$, then $\mathcal{S}$ defined as in (\ref{enlargedsparse}) is $\eta_2$-sparse for some $\eta_2=\eta_2(\eta_1,t,n)\in (0,1)$.  

 \end{lemma}

\begin{proof}
    Since $\mathcal{S}'$ is $\eta_1$-sparse there is $C'(\eta_1)>0$ such that for all $Q\in \mathcal{D}$
    $$\sum_{R'\in \mathcal{S}'\colon R'\subset Q}\mu(R')\leq C'\mu(Q).$$

    Now for any $Q\in \mathcal{D}$ one can write
    \begin{equation}
        \begin{split}
            \sum_{R\in \mathcal{S}\colon R\subseteq Q} \mu(R)\leq &\sum_{k=0}^{t} \sum_{\substack{R\in \mathcal{S}\colon \\ R^{(k)}\in \mathcal{S}',R\subseteq Q }}\mu(R)\\
            =&\sum_{k=0}^{t} \left(\sum_{\substack{R\in \mathcal{S}\colon \\ R^{(k)}\in \mathcal{S}',R^{(k)}\subseteq Q }}\mu(R)+\sum_{\substack{R\in \mathcal{S}\colon \\ R^{(k)}\in \mathcal{S}',R\in \mathcal{D}^{(k-1)}(Q) }}\mu(R)\right)\\
            \leq & \sum_{k=0}^{t} \left(\sum_{\substack{R\in \mathcal{S}\colon \\ R^{(k)}\in \mathcal{S}',R^{(k)}\subseteq Q }}\mu(R^{(k)})+\mu(Q)\#\{\mathcal{D}^{(k-1)}(Q)\}\right)\\
            \leq & (t+1)C_{t,n}(C'+1)\mu(Q)
        \lesssim_{t,n,\eta_1}\mu(Q).
        \end{split}
    \end{equation}
    
\end{proof}
 
    Let $\mathcal{B}(Q_0)=\{Q_j\}_j$. To prove the sparse domination, we write
    \begin{equation}\label{eqn:sparse_decomp}
        T^{Q_0}f(x) = S_1(x) + S_2(x) + S_3(x) + \sum_{j}T^{Q_j}(f\1_{Q_j})(x)\1_{Q_j}(x)
    \end{equation}
    where 
    \[S_1(x) =  T^{Q_0}f(x)\1_{Q_0 \setminus \bigcup \mc{B}(Q_0)}(x),\] \[ S_2(x) = \sum_{j}\sum_{Q \in \mc{D}(Q_0) : \sq{Q}_j \subsetneq Q}T_Qf(x)\1_{Q_j}(x), \] \[S_3(x) = \sum_{j}T_{\sq{Q}_j}f(x)\1_{Q_j}(x)\]
    and $\sq{Q}_j$ is the smallest cube in $\sq{\mc{D}}$ strictly containing $Q_j$. Trivially, 
    \[\text{dist}(Q_j, \sq{Q}_j) \leq t + 1.\]
    Observe that for $x\in Q_0 \setminus \bigcup \mc{B}(Q_0)$, we have that 
    \[\sum_{Q \in \mc{D}(Q_0) : J \subsetneq Q}T_Qf(x) \in C\dang{f}_{Q_0}\]
    for all cubes $J \in \sq{\mc{D}}$ containing $x$. This implies that 
    \[S_1(x) \in C\dang{f}_{Q_0}.\]
    Similarly, we have 
    \[\sum_{Q \in \mc{D}(Q_0) : \sq{Q}_j \subsetneq Q}T_Qf \in C\dang{f}_{Q_0}\]
    on $\sq{Q}_j$ by the maximality in the stopping time construction. That combined with the disjointedness of the $Q_i$'s 
 implies $S_2(x) \in C \dang{f}_{Q_0}$. It remains to estimate $S_3$.

 For all $j$, one has that 
 $$T_{\sq{Q}_j}f(x)=\langle f,h_{S_j}\rangle h_{T_j}(x),$$ where $S_j\in \mathcal{D}_s(\sq{Q}_j)$ and $T_j\in \mathcal{D}_t(\sq{Q}_j)$. We then have

 \begin{equation} \label{E: refer here for L1 norm}
 \begin{split}
 T_{\sq{Q}_j}f(x)\1_{Q_j}(x)=&\langle f,\sum_{J\in \text{ch}(S_j)} \alpha_J \1_J\rangle \sum_{K\in \text{ch}(T_j)} \alpha_K \1_{K\cap Q_j}(x)\\
=&\sum_{J\in \text{ch}(S_j)} \sum_{K\in \text{ch}(T_j),K\subset Q_j}\alpha_J\mu(J)\alpha_K\mu(K) \ang{f}_{J} \frac{\1_K}{\mu(K)} 
 \end{split}
 \end{equation}

Observe that for any $J,K$ showing up in the double sum in (\ref{E: refer here for L1 norm}) one has $J\in \mathcal{D}_{s+1}(\sq{Q}_j)$ and $K\in \mathcal{D}_{t+1}(\sq{Q}_j)$ so in particular $\text{dist}(J,K)\leq s+t+2$. So, from the balanced property one has $m(J)\sim m(K)$.
Since $K\in \mathcal{D}_{t+1}(\tilde{Q}_j)$ and $K\subset Q_j$, one has $K \subset Q_j \subsetneq \sq{Q}_j$ so  we must have $K \in \bigcup_{l=0}^{t} \mathcal{D}_l(Q_j)$, and
\begin{equation}\label{eqn:rank_bound_2}
        K \in \mc{D}^t(Q_j)\subset \mathcal{S}.
    \end{equation}

    In other words, the cubes $K$ appearing in (\ref{E: refer here for L1 norm}) are automatically in our sparse collection $\mathcal{S}$. That need not be the case for $J$ so we need to analyze those more carefully. Fix $J\in \text{ch}(S_j)$.

  If $\ang{{f}}_{J} \in C\dang{f}_{Q_0}$, we use \eqref{E: coefficientcontrolbym(Q)} and balanced property to say that 
    $$|\alpha_J|\mu(J)\leq\sqrt{m(\hat{J})}\lesssim \sqrt{m(J)}.$$ Hence $|\alpha_J|\mu(J)|\alpha_K|\mu(K)\lesssim \sqrt{m(J)}\sqrt{m(K)}\sim m(K)$.

    \begin{align*}
         \sum_{K\in \text{ch}(T_j)\colon K\subset Q_j}\alpha_J\mu(J)\alpha_K\mu(K)\ang{{f}}_{J} \frac{\1_K}{\mu(K)} & \in C \dang{f}_{Q_0} \sum_{K\in \text{ch}(T_j)\colon K\subset Q_j} m(K)  \frac{\1_K}{\mu(K)} \\
        & \in C\dang{f}_{Q_0}\1_{Q_j},
    \end{align*}
     where we used \eqref{E: misboundedbymeasure}.
    
   If $\ang{{f}}_{J} \notin C\dang{f}_{Q_0}$, then either $J \in \mc{B}_1(Q_0)$ or $J\subsetneq J_1$, for some $J_1\in \mathcal{B}_1(Q_0)$.

    If $J \in \mc{B}_1(Q_0)$, then using \eqref{E: coefficientcontrolbym(Q)},
\begin{equation}\label{goodinclusion}
        \alpha_J\mu(J)\alpha_K\mu(K) \ang{f}_{J} \frac{\1_K}{\mu(K)} \in C\dang{f}_{J} \frac{\1_K}{\mu(K)}\sqrt{m(J)}\sqrt{m(K)},\text{ with } J,K\in \mathcal{S},
    \end{equation}
 and we stop there. 

 Now for the case when $J$ is strictly contained in some cube $J_1 \in \mc{B}_1(Q_0)$, we claim that $\text{dist}(J, J_1) \leq s-1$. Indeed, since we had 
    \[J^{(s+1)} = \sq{Q}_j,\]
    it follows that $J_1$ intersects $\sq{Q}_j$ and so we must have $J_1 \subsetneq \sq{Q}_j$ (otherwise $Q_j \subsetneq \sq{Q}_j \subset J_1$ and thus would contradict the disjointedness of cubes in $\mc{B}_1(Q_0)$). So the claim follows from
    \[J\subsetneq J_1 \subsetneq \sq{Q}_j = J^{(s+1)}.\]

    If $J\in \mathcal{B}(J_1)$ we are fine since that will imply $J\in \mathcal{S}$ and we will have (\ref{goodinclusion}). If not, then either $\ang{f}_{J}\in C\dang{f}_{J_1}$ or $J\subsetneq J_2$, for some $J_2\in \mathcal{B}(J_1)\subset \mathcal{B}_2(Q_0)$, with $\text{dist}(J,J_2)\leq s-2$. Iterating this process finitely many times we can find $S \in \bigcup_{k=1}^{s}\mathcal{B}_k(Q_0)$ containing $J$ and satisfying
    \[\ang{f}_{J} \in C \dang{f}_{S}.\]
    Therefore 
    \begin{align*}
        \alpha_J\mu(J)\alpha_K\mu(K) \ang{f}_{J} \frac{\1_K}{\mu(K)} \in & C \dang{f}_{S} \frac{\1_K}{\mu(K)}\sqrt{m(J)}\sqrt{m(K)}\\
        \subseteq  & C\dang{f}_{S} \frac{\1_K}{\mu(K)}\sqrt{m(S)}\sqrt{m(K)}.
    \end{align*}
   We have thus shown that, for all $x \in Q_0$

     \[Tf(x) \in  C \dang{f}_{Q_0}  + \sum_{j}T^{Q_j}(f\1_{Q_j})(x)\1_{Q_j}(x)  + C\sum_{\substack{\substack{J \in \bigcup_{k=1}^s \mathcal{B}_k(Q_j) \\ K \in \D^t(Q_j): \\ \dd(J, K) \leq s+t + 2 }}}\dang{f}_J \frac{\1_{K}(x)}{\mu(K)}\sqrt{m(J)}\sqrt{m(K)} \]
and Theorem \ref{T:thm A 1st part} thus follows from iterating this construction at each level of the sparse family.      

\end{proof}

\subsection{Convex body domination for $L^1$ normalized shifts}
As we mentioned before, the main reason for the failure of usual sparse domination for Haar shifts, even in the balanced setting, seems to be the worse behaviour of the kernel $K_Q$ of $T_Q$. Indeed, if we assume that the coefficients of the shift are just bounded, we only get  $$\|K_Q\|_{L^\infty(\mu)} \lesssim_{\mu,s,t} \frac{1}{m(Q)}, \quad \text{ for any $Q \in \D $,}$$
instead of $\frac{1}{\mu(Q)}$ on the right-hand side, which is the usual upper bound in the doubling setting. 

In what follows, we work with a general Radon measure $\mu$ on $\R^d$, not necessarily balanced, such that $0 < \mu(Q) < \infty$ for every $Q \in \D$ and such that each quadrant has infinite measure. 
First we introduce the nonhomogeneous Calderón-Zygmund decomposition. To state it we need the nonhomogeneous $\BMO$ norm:
$$
\|f\|_{\BMO(\mu)} =\sup_{Q\in\D} \frac{1}{\mu(Q)} \int_Q|f - \LL f\RR_{\widehat{Q}}| d\mu.
$$

\begin{lemma}\cite{CPW} \label{L: CZD}
Let $f:\R^n\rightarrow \R$ with $f \in L^1(\mu)$ supported in $Q_0 \in \D$. Then, for every $\lambda>0$ there exist functions $g,b$ such that $f=g+b$ and the following holds
\begin{enumerate}
    \item There exists a family of pairwise disjoint intervals $\{Q_k\}_k \subset \D(Q_0)$ such that \[b= \sum_{k \in \N} b_k; \quad \quad b_k= f\1_{Q_k}- \LL f\1_{Q_k}\RR_{\widehat{Q_k}}\1_{\widehat{Q_k}};\]
    In particular, for every $k$, $\|b_k\|_{L^1(\mu)} \lesssim \int_{Q_k} |f| d\mu$ and $b_k$ has zero mean on $\widehat{Q_k}$. Finally, the cubes are selected via the stopping time argument for the dyadic maximal function, therefore $\LL |f| \RR_{Q_k} > \lambda$.

    \item We have that $g \in L^p(\mu)$ for every $1 \leq p < \infty $ and $\|g\|^p_{L^p(\mu)} \lesssim_p \lambda^{p-1}\|f\|_{L^1(\mu)}$. Moreover, $g \in \BMO(\mu)$ and $\|g\|_{\BMO} \leq \lambda.$
\end{enumerate}
\end{lemma}

\begin{lemma} \label{T:weak (1,1) for L^1 norm}
Let $\mu$ be as before, $\mathscr{H}$ a standard Haar system, and $T$ be a $L^1$ normalized shift of complexity $(s,t)$ as in Definition \ref{D: L^1 normalized} with $d=1$ (scalar valued). Then 
$$\|T\|_{L^1(\mu) \to L^{1, \infty}(\mu)} \lesssim C(s,d,\mu),$$
where the constant $C(s,d,\mu)$ is linear in the complexity. 
\end{lemma}

\begin{proof}

 As before, write $T=\sum_{Q \in \D} T_Q$, where $T_Q$ has kernel 
 $$ \|K_{Q}\|_{\infty}\lesssim_{\mu}\frac{1}{\mu(Q)}.$$

 Let $\lambda>0$, $f=g+b$ the Calder\'on-Zygmund decomposition of $f$ at height $\lambda.$ Using \cref{L: CZD} and $L^2$ boundedness of Haar shifts
 \begin{align*}
     \mu \left( \left\{ x: |Tf(x)| > \lambda \right\}\right) \leq &  \mu(\left\{ x: |Tg(x)| > \lambda/2 \right\})+ \mu(\left\{ x: |Tb(x)| > \lambda/2 \right\}) \\ \leq & \frac{C}{\lambda}\|f\|_{L^1(\mu)}  + \mu( \left\{ x: |Tb(x)| > \lambda/2 \right\}),
 \end{align*}

 Hence we only have to estimate the second term. Write 
 \[Tb_j(x)= T^{Q_j}b_j(x) + \sum_{Q_j \subsetneq Q} T_Qb_j(x)\]
 so that 
 \[Tb(x)= \sum_{j}T^{Q_j}b_j(x) + \sum_j \sum_{Q_j \subsetneq Q} T_Qb_j(x)= T_1b(x)+ T_2b(x).\]
 Notice that $\mathrm{supp}(T_1b) \subseteq \bigcup_j Q_j$, so that  
 $$\mu\left(\left\{x: |T_1b(x)| > \frac{\lambda}{4} \right\}\right) \leq \sum_j\mu( Q_j) \leq \frac{\|f\|_{L^1(\mu)}}{\lambda}.$$
 For a fixed $j$, notice that $\sum_{Q_j \subsetneq Q} T_Qb_j(x)$ has a finite number of terms, depending linearly on the complexity of the shift. Indeed, we have that
 $T_Qb_j(x)$ vanishes unless $Q \subseteq Q_j^{(s+1)}$.
 All together we have that 
 \[T_2b(x)= \sum_{j} \sum_{Q_j\subsetneq Q\subseteq Q_j^{(s+1)}}T_Qb_j(x),\]
 and 
 \begin{align*}
     \mu\left( \left\{ x: |T_2b(x)| > \frac{\lambda}{4} \right\}\right) \leq & \frac{4}{\lambda} \sum_j \sum_{ Q_j\subsetneq Q\subseteq Q_j^{(s+1)}} \|T_Qb_j\|_{L^1(\mu)} \\ \leq & \frac{4}{\lambda} \sum_j \sum_{ Q_j\subsetneq Q\subseteq Q_j^{(s+1)}} \|T_Q\|_{L^1(\mu) \to L^1(\mu)}\|b_j\|_{L^1(\mu)}.
 \end{align*}
 
 Since $\|T_Q\|_{L^1(\mu) \to L^1(\mu)}= \text{ess} \sup_y \|K_Q(\cdot,y )\|_{L^1(\mu)} \lesssim_{\mu} 1$
 we can conclude that 
  \begin{align*}
     \mu\left(\left\{ x: |T_2b(x)| > \frac{\lambda}{4} \right\}\right)\lesssim_{\mu} & \frac{1}{\lambda} C(s) \sum_j \|b_j\|_{L^1(\mu)} \lesssim_{\mu} \frac{1}{\lambda}C(s) \|f\|_{L^1(\mu)},
 \end{align*}
 where $C(s)=\#\{Q\in \mathcal{D}\colon Q_j\subsetneq Q\subseteq Q_j^{(s+1)}\}=s.$
 This concludes the proof.
\end{proof}
Now we prove the second part of \cref{thm: sparseforvectorhaar}.
\begin{theorem} \label{T:thm A 2nd part}
Let $\mathscr{H}=\{h_Q\}_{Q \in \D}$ be a Haar system with respect to $\mu$. For any $f \in L^1(\R^n; \R^d)$ supported on a dyadic cube $Q_0$, there is an $\eta$-sparse family $\mc{S}$ and a positive constant $C=C(d,s,t)$ so that 
\[Tf(x) \in C\sum_{Q \in \mc{S}}\dang{f}_Q \1_{Q}(x) \quad \text{ on $Q_0$. }\]
Also, the constant $C=C(d,s,t)$ depends linearly on the complexity.
\end{theorem}

\begin{proof} 
The proof goes exactly as the one of \cref{T:thm A 1st part}, but we have to modify the stopping condition \eqref{E: stop2}, while Lemma \ref{T:weak (1,1) for L^1 norm} combined with \cref{lem:convex_body_components} allows us to use \eqref{E: stop1}. Notice that \eqref{E: stop2} essentially plays a role only when we have to deal with the term $T_{\widetilde{Q_j}}f(x) \1_{Q_j}(x)$ as in \eqref{E: refer here for L1 norm}, where $Q_j$ is a selected cube and $\widetilde{Q_j}$ is its first ancestor in $\widetilde{\D}$. Our goal is to show that, using the extra properties of $L^1$ normalized shifts, we have that
\begin{equation} \label{E: stopping term L^1 norm}T_{\widetilde{Q_j}}f(x) \1_{Q_j}(x) \in C \dang{f}_{Q_0}\1_{Q_j}(x).\end{equation}
For an interval $Q' \in \D(Q_0)$, let $\mc{Q}_1(Q')$ be the collection of intervals satisfying \eqref{E: stop1}, for which we have as before 
\[\sum_{J \in \mc{Q}_1(Q')}\mu(J) \leq \frac{1}{4}\mu(Q').\]
We substitute \eqref{E: stop2} as follows. Let $\{e_i\}_{i = 1}^{d}$ be an orthonormal basis associated to the John ellipsoid $\mc{E}$ of $\dang{f}_{Q'}$, and let $\mc{Q}_2(Q')$ be the collection of maximal intervals $J \in \D(Q')$ such that \begin{equation} \label{E: new stop 2}
  \ang{|f \cdot e_i|}_J > \frac{C}{d}\ang{|f \cdot e_i|}_{Q'}, \quad \text{ for some $i=1, \dots,d$}.  
\end{equation}
Let us comment about this twist in the proof. The vector stopping condition \eqref{E: stop2} also carried information about the components of the selected vector averages: indeed, the selected cubes $J$ verified $\ang{f}_J \notin C \dang{f}_{Q'}$, that implies by \cref{lem:convex_body_components} \[ | \ang{ f \cdot e_i }_J| > \frac{C}{d} \ang{|f \cdot e_i|}_{Q'}, \quad \text{ for some $i=1, \dots, d$}.\]
Informally, as of course this last condition implies \eqref{E: new stop 2}, we are now selecting more cubes than before and the ones that we didn't select will behave better. Nevertheless, using the weak $(1,1)$ for the dyadic maximal function and choosing $C$ big enough as before, we still have that
\[\sum_{J \in \mc{Q}_2(Q')}\mu(J) \leq \frac{1}{4}\mu(Q').\]

Indeed, denoting $E(Q')= \bigcup_{J \in \mc{Q}_2(Q')}J$, if $E_i(Q')$ is the union of maximal intervals selected by \eqref{E: new stop 2} for a fixed $i \in \{ 1, \dots, d\}$, we have   
\[\sum_{J \in \mc{Q}_2(Q')}\mu(J) = \mu(E(Q')) \leq \sum_{i=1}^d \mu( E_i(Q')) \leq \|M_D\|_{L^1 \to L^{1, \infty}}\frac{d^2}{C} \mu(Q') \leq \frac{1}{4} \mu(Q'),\] 
choosing $C \geq  4d^2\|M_{\mathcal{D}}\|_{L^1 \to L^{1, \infty}}$. All together, if $\mc{B}(Q'):=\mc{Q}_1(Q') \cup \mc{Q}_2(Q')$, we still have that 
 \[\sum_{J \in \mc{B}(Q')}\mu(J) \leq \frac{1}{2}\mu(Q').\]
Now we show how the modified stopping condition helps us to prove \eqref{E: stopping term L^1 norm}. As $T$ is $L^1$ normalized
\begin{align*}
    T_{\sq{Q}_j}f(x)\1_{Q_j}(x)= & c_{\widetilde{Q_j}} \langle f,h_{S_j}\rangle h_{T_j}(x)\1_{Q_j}(x) \\ = &  c_{\widetilde{Q_j}} \mu(\widetilde{Q_j}) \|h_{S_j}\|_{L^\infty(\mu)} \left\langle \frac{h_{S_j}}{\|h_{S_j}\|_{L^\infty(\mu)}}f\right\rangle_{\widetilde{Q_j}}  \,h_{T_j}(x)\1_{Q_j}(x) 
\end{align*}
 If we prove that 
\begin{equation} \label{E: final step L^1 norm} v:= \left\langle \frac{h_{S_j}}{\|h_{S_j}\|_{\infty}}f\right\rangle_{\widetilde{Q_j}}  \in C(d) \dang{f}_{Q_0},\end{equation}
then, we can conclude  
 \begin{align*}
   T_{\sq{Q}_j}f(x)\1_{Q_j}(x) \in  \ & C(d) c_{\widetilde{Q_j}} \mu(\widetilde{Q_j}) \|h_{S_j}\|_{\infty}\|h_{T_j}\|_{\infty}\dang{f}_{Q_0} \1_{Q_j}(x) \\
   \in 
   \ & C(d) \dang{f}_{Q_0} \1_{Q_j}(x),
 \end{align*}   
where we used that from  $\|K_{\widetilde{Q_j}}\|_{\infty}\lesssim \mu(\widetilde{Q_j})^{-1}$,

$$|c_{\widetilde{Q_j}}| \lesssim 
\frac{1}{\mu(\widetilde{Q_j}) \|h_{S_j}\|_{\infty}\|h_{T_j}\|_{\infty}}.$$

 We are left with proving \eqref{E: final step L^1 norm}. Let $\{e_i\}_{i = 1}^{d}$ be an orthonormal basis associated to the John ellipsoid $\mc{E}$ of $\dang{f}_{Q_0}$. By the stopping condition we have that \[ \ang{|f \cdot e_i|}_{\widetilde{Q_j}} \leq \frac{C}{d}\ang{|f \cdot e_i|}_{Q_0}, \quad \text{ for every $i=1, \dots,d$}. \] 
 For every $i= 1, \dots, d$, we have that
\begin{align*}|v \cdot e_i|\leq \ & \left\langle |f \cdot e_i| \frac{|h_{S_j}|}{\|h_{S_j}\|_{\infty}}\right\rangle_{\widetilde{Q_j}} \\
\leq \ &   \LL |f \cdot e_i| \RR_{\widetilde{Q_j}} \\
\leq \ & \frac{C}{d} \LL |f \cdot e_i| \RR_{Q_0} \\
\lesssim \ & d  \ \LL |f \cdot e_i| \RR_{Q_0}.
\end{align*}
Applying \cref{lem:convex_body_components} we get \eqref{E: final step L^1 norm} and \eqref{E: stopping term L^1 norm} follows.

\end{proof}

\section{Matrix weighted estimates} \label{section 4}
In this section, we prove \cref{MatrixWeightedThm} and \cref{GeneralizedCarlesonEmbedding}. We will be working with an arbitrary Radon measure $\mu$ on $\R^n$ and a dyadic lattice $\mc{D}$ of $\R^n$. Fix $f \in L^p(\R^n; \R^d)$ and $g \in L^{p'}(\R^n; \R^d)$. To bound the sparse form 
\[\mc{A}_{\mc{S}}(V^{1/p'}f, W^{1/p}g) = \sum_{Q \in \mc{S}} \dashint_{Q}\dashint_{Q} |V^{1/p'}(x)f(x) \cdot W^{1/p}(y)g(y)| \, d\mu(x) \,d\mu(y) \, \mu(Q),\]
we split the expression up into two sparse operators. Consider the reducing operators $\mc{W}_Q$ and $\mc{V}_Q$ as in Proposition \ref{reducing operators} so that
\[|\mc{W}_Qe| \sim \left(\dashint_{Q}|W^{1/p}e|^p \, d\mu\right)^{1/p}, \quad |\mc{V}_Qe| \sim \left(\dashint_{Q}|V^{1/p'}e|^{p'} \, d\mu\right)^{1/p'},\,\text{ for all }e\in \R^d.\]

We omit the $p$ and $p'$ from the notation of the reducing operators for convenience. We have that
\begin{align*}
    |V^{1/p'}(y)f(y) \cdot W^{1/p}(x)g(x)| & = |\mc{V}_Q\mc{V}_Q^{-1}V^{1/p'}(y)f(y) \cdot \mc{W}_{Q}\mc{W}_{Q}^{-1}W^{1/p}(x)g(x)| \\
    & = |\mc{W}_Q\mc{V}_{Q}\mc{V}_Q^{-1}V^{1/p'}(y)f(y) \cdot \mc{W}_{Q}^{-1}W^{1/p}(x)g(x)| \\
    & \leq \bm{|}\mc{W}_Q\mc{V}_Q \bm{|} \, |\mc{V}_Q^{-1}V^{1/p'}(y)f(y)| \, |\mc{W}_Q^{-1}W^{1/p}(x)g(x)| \\
    & \lesssim_d [W, V]_{A_p}^{1/p} \, |\mc{V}_Q^{-1}V^{1/p'}(y)f(y)| \, |\mc{W}_Q^{-1}W^{1/p}(x)g(x)|.
\end{align*}
Here we used that the reducing operators are symmetric and \cref{reducing operators}. Therefore
\begin{align*}
    \mc{A}_{\mc{S}}(V^{1/p'}f, W^{1/p}g) & \lesssim_d [W, V]_{A_p}^{1/p} \sum_{Q \in \mc{S}}\dashint_{Q}\dashint_{Q}|\mc{V}_Q^{-1}V^{1/p'}(y)f(y)| \, |\mc{W}_Q^{-1}W^{1/p}(x)g(x)| d\mu(x) \, d\mu(y) \, \mu(Q) \\
    & = [W, V]_{A_p}^{1/p} \sum_{Q \in \mc{S}} \ang{|\mc{V}_Q^{-1}V^{1/p'}f|}_Q\ang{|\mc{W}_Q^{-1}W^{1/p}g|}_Q\mu(Q) \\
    & \leq [W, V]_{A_p}^{1/p} \|S^V_{p} f\|_{L^{p}} \|S^W_{p'}g\|_{L^{p'}}
\end{align*}
where 
\begin{equation}\label{eqn:sparse_square_1}
    S_{p}^V f(x) = \left(\sum_{Q \in \mc{S}}  \ang{\bv{\mc{V}_Q^{-1}V^{1/p'}}\,|f|}_Q^p \1_Q(x)\right)^{1/p}
\end{equation}
and 
\begin{equation}\label{eqn:sparse_square_2}
    S^{W}_{p'} g(x) = \left(\sum_{Q \in \mc{S}}  \ang{\bv{\mc{W}_Q^{-1}W^{1/p}} \, |g|}_Q^{p'} \1_Q(x)\right)^{1/p'}.
\end{equation}
These operators depend only on the norm of the vector function $f$. Thus to derive bounds on $L^p(\R^n; \R^d)$, it suffices to derive bounds on $L^p(\R^n; \R)$. We will show that
\begin{lemma}\label{lem:square_function_estimate}
    Let $W$ be a matrix weight satisfying the $A_{p,\infty}^{sc}$ condition and let $\mc{S}$ be an $\eta$-sparse family. Define $S_{p'}^{W}$ by \cref{eqn:sparse_square_2}. Then 
    \[\|S_{p'}^{W}g\|_{L^{p'}} \lesssim_{d, \eta} [W]_{A_{p, \infty}^{sc}}^{1/p'}\|g\|_{L^{p'}}.\]
\end{lemma}
This lemma will follow from \cref{GeneralizedCarlesonEmbedding}. To motivate this theorem, we recall the dyadic Carleson Embedding Theorem.
\begin{theorem}[Dyadic Carleson Embedding Theorem]\label{thm:carleson_embedding}
    Let $\mu$ be a Radon measure on $\R^n$ and let $\{\alpha_Q\}_{Q \in \mc{D}}$ be a sequence of real numbers indexed by $\mc{D}$. Then for $1 < p < \infty$,
    \begin{gather}
            \sum_{Q \in \mc{D}}\ang{|f|}_Q^{p}\alpha_Q \leq C_1 \|f\|_{L^{p}}^{p} \quad \text{for all} \quad  f \in L^p(\mu) \\ 
            \nonumber \text{if and only if}  \\ 
            \frac{1}{\mu(P)}\sum_{Q \in \mc{D}(P)}\alpha_Q \leq C_2  \quad \text{for all} \quad P \in \mc{D}
    \end{gather}
    Moreover, if $C_1$ and $C_2$ are the best constants in the above inequalities, then $C_2 \leq C_1 \lesssim_p C_2$.
\end{theorem}

If $w$ is a weight, (i.e. $w > 0$ a.e. and $w \in L^1_{loc}(\mu)$), then by replacing $\mu$ with $w \, d\mu$ and $f$ with $fw^{\frac{-1}{p}}$ in \cref{thm:carleson_embedding}, we have the alternate (weighted) formulation of the Carleson Embedding Theorem:
\begin{gather}
    \sum_{Q \in \mc{D}}\ang{w^{1/p'}f}_Q^{p}\alpha_Q \leq C_1 \|f\|_{L^{p}}^{p} \quad \text{for all} \quad  f \in L^p(\mu) \\ 
    \nonumber \text{if and only if}  \\ 
    \frac{1}{\mu(P)}\sum_{Q \in \mc{D}(P)}\ang{w}_Q^{p}\alpha_Q \leq C_2 \ang{w}_P \quad \text{for all} \quad P \in \mc{D}
\end{gather}

In the scalar case (i.e. if $d = 1$) \cref{lem:square_function_estimate} follows from this formulation as the reader will easily check. However, in the vector case, the operator $S_{p'}^W$ depends in a more complicated way on the underlying weight $W$ so \cref{thm:carleson_embedding} doesn't apply. This motivates the generalization of the Carleson Embedding Theorem. The idea is that we take the weighted formulation above, and instead of having only one weight $w$, we have a collection of weights $\{w_Q\}_{Q \in \mc{D}}$ indexed by the dyadic cubes. We restate \cref{GeneralizedCarlesonEmbedding} here.

\begin{theorem}\label{thm:carleson_embedding_general}
    Let $\{w_Q\}_{Q \in \mc{D}}$ be a collection of scalar weights indexed by $\mc{D}$ so that each $w_Q$ is localized to $Q$ (i.e. $w_Q \in L^1_{loc}(\mu)$ and $w_Q > 0$ $\mu$-almost everywhere on $Q$) and so that the following property is satisfied:
    \begin{equation}\label{eqn:weight_condition}
        \|w_P w_{Q}^{-1}\|_{L^\infty(Q)} \leq A\ang{w_{P}}_{Q}\ang{w_{Q}}_{Q}^{-1} \quad \text{for all} \quad Q,P \in \mc{D}, \quad Q \subset P
    \end{equation}
    Let $\{\alpha_Q\}_{Q \in \mc{D}}$ be a sequence of real numbers indexed by $\mc{D}$. Then for $1 < p < \infty$,
    \begin{gather}
            \sum_{Q \in \mc{D}}\ang{w_Q^{1/p'}f}_Q^{p}\alpha_Q \leq C_1 \|f\|_{L^{p}}^{p} \quad \text{for all} \quad  f \in L^p(\mu) \label{eqn:carleson_condition_1} \\ 
            \nonumber \text{if and only if}  \\ 
            \frac{1}{\mu(P)}\sum_{Q \in \mc{D}(P)}\ang{w_P}_{Q}\ang{w_Q}_Q^{p-1}\alpha_Q \leq C_2 \ang{w_P}_P \quad \text{for all} \quad P \in \mc{D} \label{eqn:carleson_condition_2}
    \end{gather}
    Moreover, if $C_1$ and $C_2$ are the best constants in the above inequalities, then, 
    \[A^{-(p-1)}C_2 \leq C_1 \lesssim_{p} A^{1 + 1/p'}C_2.\]
\end{theorem}

\begin{remark}
    \cref{thm:carleson_embedding_general} can be seen as a variable weight generalization of the weighted Carleson Embedding Theorem given above. Indeed, if $w$ is a weight, then by setting $w_Q = w$ for all $Q$ then $\cref{eqn:weight_condition}$ holds trivially with $A = 1$ and \cref{thm:carleson_embedding_general} recovers the weighted Carleson Embedding Theorem. 
\end{remark}
Let's take a moment to interpret the compatibility condition \cref{eqn:weight_condition}. Let $\{w_Q\}$ be a sequence of weights satisfying \cref{eqn:weight_condition}. \par
First observe that the reverse inequality in \eqref{eqn:weight_condition} always holds by Hölder's inequality: \[\ang{w_{P}}_{Q}\ang{w_{Q}}_{Q}^{-1} = \frac{\int_Q w_P}{\int_Q w_Q} =  \frac{\int_Q w_P w_Q w_Q^{-1}}{\int_Q w_Q} \leq \|w_P w_Q^{-1}\|_{L^\infty(Q)}.\] 
Hence, it can be interpreted as a reverse Hölder-type condition for the family $\{w_Q\}_Q$. In particular, $A \geq 1$. \par
By rewriting \cref{eqn:weight_condition}, we see that 
\[\frac{w_P}{\ang{w_P}_Q} \leq A \frac{w_Q}{\ang{w_Q}_Q}\]
on $Q$ which intuitively says that, up to appropriate scaling, $w_P$ isn't too large on $Q$ compared to $w_Q$ when $Q \subset P$.
\par

To prove the \cref{thm:carleson_embedding_general}, we will need the following elementary lemma about expanding a sum $(\sum a_j)^p$ when $p$ is not necessarily an integer.

\begin{lemma}[\cite{hytonen2012}]\label{lem:expanding_sum}
    Let $m \in \N$, $\gamma \in [0, 1]$, and $p=m+\gamma$. Then for any sequence $\{a_i\}$ of nonnegative real numbers, we have 
    \[\left(\sum_i a_i\right)^p \leq (m + 1)\sum_{i_1, ..., i_m}a_{i_1} \cdots a_{i_m}\left(\sum_{j \leq \min\{i_1, ..., i_m\}} a_j\right)^\gamma.\]
\end{lemma}

\begin{proof}[Proof of \cref{thm:carleson_embedding_general}]
    We start by making a few reductions to make the notation more convenient to work with. Note that neither the hypothesis nor the conclusion of the theorem involve the values of $w_Q$ outside of $Q$. Therefore we assume that $w_Q$ is supported in $Q$.
    Next, note that the hypothesis and conclusion of the theorem are invariant under the scaling $w_Q \to C_Qw_Q$. Therefore we may normalize $w_Q$ by replacing $w_Q$ with $w_Q / \ang{w_Q}_{Q}$ so that $\ang{w_Q}_Q = 1$ for all $Q$. We also replace $\alpha_Q$ with $\alpha_Q\mu(Q)$. Then \cref{eqn:weight_condition} becomes 
    \begin{equation}\label{eqn:weight_condition_adj}
        \|w_P w_{Q}^{-1}\|_{L^\infty(Q)} \leq A\ang{w_{P}}_{Q} \quad \text{for all} \quad Q,P \in \mc{D}, \quad Q \subset P
    \end{equation}
    and the conclusion of the theorem becomes
    \begin{gather}
           \sum_{Q \in \mc{D}}\ang{w_Q^{1/p'}f}_Q^{p}\mu(Q)\alpha_Q \leq C_1 \|f\|_{L^{p}}^{p} \quad \text{for all} \quad  f \in L^p(\mu) \label{eqn:carleson_condition_adj_1} \\ 
            \nonumber \text{if and only if}  \\ 
            \frac{1}{\mu(P)}\sum_{Q \in \mc{D}(P)}\ang{w_P}_{Q}\mu(Q)\alpha_Q \leq C_2 \quad \text{for all} \quad P \in \mc{D} \label{eqn:carleson_condition_adj_2}
    \end{gather}
    and $C_2A^{-(p-1)} \leq C_1 \lesssim_p A^{1 + 1/p'} C_2$. \par
    We first prove the easier implication \cref{eqn:carleson_condition_adj_1} $\Rightarrow$ \cref{eqn:carleson_condition_adj_2} and $C_2 \leq A^{p - 1} C_1$. Indeed, assume \cref{eqn:carleson_condition_adj_1} and let $P \in \mc{D}$. For $Q \in \mc{D}(P)$, we have that, on $Q$,
    \[w_P = w_P^{1/p}w_P^{1/p'} \leq A^{1/p'}w_{P}^{1/p}w_Q^{1/p'}\ang{w_P}_{Q}^{1/p'}\]
    by \cref{eqn:weight_condition_adj}. Averaging over $Q$ gives 
    \[\ang{w_P}_Q^{1/p} \leq A^{1/p'}\ang{w_{P}^{1/p}w_Q^{1/p'}}_Q.\]
    Therefore by setting $f = w_P^{1/p}\1_P$, we have 
    \begin{align*}
        \sum_{Q \in \mc{D}(P)}\ang{w_P}_{Q}\alpha_Q\mu(Q) & \leq A^{\frac{p}{p'}}\sum_{Q \in \mc{D}}\ang{w_{Q}^{1/p'}f}^{p}_Q \alpha_Q \mu(Q) \\
        & \leq A^{\frac{p}{p'}}C_1\|f\|_{L^{p}}^{p} \\ 
        & = A^{\frac{p}{p'}}C_1\int_{P}w_P \, d\mu  \\
        & = A^{\frac{p}{p'}}C_1\mu(P)
    \end{align*}
    where the last equality holds by the normalization of $w_Q$. Therefore \cref{eqn:carleson_condition_adj_1} $\Rightarrow$ \cref{eqn:carleson_condition_adj_2} and $C_2 \leq A^{\frac{p}{p'}}C_1 = A^{p-1}C_1$. \par 
    We now prove the harder implication \cref{eqn:carleson_condition_adj_2} $\Rightarrow$ \cref{eqn:carleson_condition_adj_1} and $C_1 \lesssim_p A^{1 + 1/p'} C_2$. Consider the operator $H$ defined on $L^p$ given by 
    \[Hf = \left\{ \ang{w_Q^{1/p'} f}_Q \right\}_{Q \in \mc{D}}.\]
    Set 
    \[\beta = \{\beta_Q\}_{Q \in \mc{D}}, \quad \beta_Q = \alpha_Q \mu(Q).\]
    It's easy to see that \cref{eqn:carleson_condition_adj_1} and $C_1 \lesssim_p A C_2$ is equivalent to 
    \begin{equation}\label{eqn:operator_bound}
        \|Hf\|_{\ell^p(\beta)}^p \lesssim_p A^{1 + 1/p'}C_2\|f\|_{L^p}^p.
    \end{equation}
    Here we are taking 
    \[\|x\|_{l^{p}(\beta)}^{p} := \sum_{Q \in \mc{D}}|x_{Q}|^{p}\beta_Q.\]
    We now aim to show \cref{eqn:operator_bound}.
    It will be more convenient to bound the adjoint operator $H^*: l^{p'}(\beta) \to L^{p'}$ which is given by 
    \[H^*x = \sum_{Q \in \mc{D}}x_Q w_Q^{1/p'}\alpha_Q.\] 
    Without loss of generality, assume $x_Q \geq 0$ for all $Q$. Write $p' = m + \gamma$ where $m = \fl{p'} \in \N$ and $\gamma = p' - \fl{p'} \leq 1$. By \cref{lem:expanding_sum}, we have that 
    \begin{align}
        \|H^*x\|_{L^{p'}}^{p'} & = \int\left(\sum_{Q \in \mc{D}}x_Qw_Q^{1/p'}\alpha_Q\right)^{p'} \, d\mu \nonumber \\
        & \lesssim_p \int \sum_{\substack{Q_1, ..., Q_{m} \in \mc{D}}} \left(\prod_{j = 1}^{m}x_{Q_j}w_{Q_j}^{1/p'}\alpha_{Q_j}\right)\left(\sum_{\substack{Q_{m+1} \in \mc{D}: \\ Q_{m+1} \supset Q_{i} \text{ for all }i}}x_{Q_{m+1}}w_{Q_{m+1}}^{1/p'}\alpha_{Q_{m+1}}\right)^\gamma \, d\mu \nonumber \\
        & = \sum_{\substack{Q_1, ..., Q_{m} \in \mc{D}}} \left(\prod_{j = 1}^{m}x_{Q_j}\alpha_{Q_j}\right)\int_{Q_1}w_{Q_1}^{1/p'}\cdots w_{Q_m}^{1/p'}\left(\sum_{\substack{Q_{m+1} \in \mc{D}: \\ Q_{m+1} \supset Q_{m}}}x_{Q_{m+1}}w_{Q_{m+1}}^{1/p'}\alpha_{Q_{m+1}}\right)^\gamma  \, d\mu \label{eqn:big_sum_1}.
    \end{align}
    Fix some cubes $Q_1, ..., Q_{m+1} \in \mc{D}$. For the summand in \cref{eqn:big_sum_1} given by cubes $Q_1, ..., Q_{m+1}$ to be nonzero, it must be the case that $Q_i$ intersects every other $Q_j$ by support considerations. Therefore for some permutation $i_1, ..., i_m$ of $1, ..., m$, we must have 
    \[Q_{i_1} \subset Q_{i_2} \subset \cdots Q_{i_m} \subset Q_{m+1}.\]
    Using this fact along with the observation that each summand in \cref{eqn:big_sum_1} is unchanged by reordering the $Q_i$ for $i \leq m$, we have that
    \[\|H^*x\|_{L^{p'}}^{p'} \lesssim_p \sum_{\substack{Q_1, ..., Q_{m} \in \mc{D}: \\ Q_1 \subset \cdots \subset Q_{m}}} \left(\prod_{j = 1}^{m}x_{Q_j}\alpha_{Q_j}\right)\int_{Q_1}w_{Q_1}^{1/p'}\cdots w_{Q_m}^{1/p'}\left(\sum_{\substack{Q_{m+1} \in \mc{D}: \\ Q_{m+1} \supset Q_{m}}}x_{Q_{m+1}}w_{Q_{m+1}}^{1/p'}\alpha_{Q_{m+1}}\right)^\gamma  \, d\mu.\]
    To bound the integral, we apply \cref{eqn:weight_condition_adj} giving
    \[w_{Q_j} \leq Aw_{Q_1}\ang{w_{Q_j}}_{Q_1} \quad \text{on} \quad Q_1\]
    for each $2 \leq j \leq m + 1$. Therefore
    \begin{align*}
        & \int_{Q_1}w_{Q_1}^{1/p'}\cdots w_{Q_m}^{1/p'}\left(\sum_{\substack{Q_{m+1} \in \mc{D}: \\ Q_{m+1} \supset Q_{m}}}x_{Q_{m+1}}w_{Q_{m+1}}^{1/p'}\alpha_{Q_{m+1}}\right)^\gamma  \, d\mu \nonumber \\
        & \leq \int_{Q_1}w_{Q_1}^{1/p'}\left(\prod_{j=2}^{m}(Aw_{Q_1}\ang{w_{Q_j}}_{Q_1})^{1/p'}\right)\left(\sum_{\substack{Q_{m+1} \in \mc{D}: \\ Q_{m+1} \supset Q_{m}}}x_{Q_{m+1}}(Aw_{Q_1}\ang{w_{Q_{m+1}}}_{Q_1})^{1/p'}\alpha_{Q_{m+1}}\right)^\gamma  \, d\mu \nonumber \\ 
        & = A^{(m - 1 + \gamma) / p'}\left(\prod_{j = 2}^{m}\ang{w_{Q_j}}_{Q_1}^{1/p'}\right)\left(\sum_{\substack{Q_{m+1} \in \mc{D}: \\ Q_{m+1} \supset Q_{m}}}x_{Q_{m+1}}\ang{w_{Q_{m+1}}}_{Q_1}^{1/p'}\alpha_{Q_{m+1}}\right)^\gamma \int_{Q_1}w_{Q_1}^{(m + \gamma) / p'} \, d\mu \nonumber
    \end{align*} 
    \begin{equation}\label{eqn:big_sum_2}
        = A^{1/p}\left(\prod_{j = 2}^{m}\ang{w_{Q_j}}_{Q_1}^{1/p'}\right)\left(\sum_{\substack{Q_{m+1} \in \mc{D}: \\ Q_{m+1} \supset Q_{m}}}x_{Q_{m+1}}\ang{w_{Q_{m+1}}}_{Q_1}^{1/p'}\alpha_{Q_{m+1}}\right)^\gamma \mu(Q_1)
    \end{equation}
    where the last equality holds from the fact that $p' = m + \gamma$ and the normalization of the weights. Let
    \[k_{Q, P} := \begin{cases}\ang{w_{P}}_Q^{1/p'}, & Q \subset P \\ 0, & Q \not \subset P\end{cases}.\]
    By \cref{eqn:weight_condition_adj}, we have 
    \begin{equation}\label{eqn:kernel_prop}
        k_{R, P} \leq A^{1/p'} \,  k_{R, Q}k_{Q, P}, \quad \text{for} \quad R \subset Q \subset P.
    \end{equation}
    \cref{eqn:big_sum_1} and \cref{eqn:big_sum_2} imply that
    \[\|H^*x\|_{p'}^{p'} \lesssim_p A^{1/p}I\]
    where
    \begin{align*}
        I & = \sum_{\substack{Q_1, ..., Q_{m} \in \mc{D}}} x_{Q_1} \left(\prod_{j = 2}^{m}x_{Q_j}k_{Q_1,Q_j}\alpha_{Q_j}\right)\left(\sum_{\substack{Q_{m+1} \in \mc{D}}}x_{Q_{m+1}}k_{Q_1, Q_{m+1}}\alpha_{Q_{m+1}}\right)^\gamma \beta_{Q_1} \\
        & = \sum_{Q_1 \in \mc{D}}x_{Q_1}\left(\sum_{Q_2 \in \mc{D}}x_{Q_2}k_{Q_1,Q_2}\alpha_{Q_{2}}\right)^{p'-1}\beta_{Q_1}.
    \end{align*}
    Then by Hölder's inequality,
    \begin{equation}\label{eqn:holders}
        I \leq \|x\|_{l^{p'}(\beta)}\|z^{p'-1}\|_{l^{p}(\beta)} =  \|x\|_{l^{p'}(\beta)}\|z\|_{l^{p'}(\beta)}^{p'-1}
    \end{equation}
    where 
    \[z_{Q_1} := \sum_{Q_2 \in \mc{D}}x_{Q_2}k_{Q_1,Q_2}\alpha_{Q_2}.\]
    Using \cref{lem:expanding_sum} again and \cref{eqn:kernel_prop}, we estimate 
    \begin{align}
        \|z\|_{l^{p'}(\beta)}^{p'} & \lesssim_p \sum_{Q_1, ..., Q_{m+1} \in \mc{D}}\left(\prod_{j = 2}^{m+1}x_{Q_j}k_{Q_1,Q_j}\alpha_{Q_j}\right)\left(\sum_{\substack{Q_{m+2} \in \mc{D}: \\ Q_{m+2} \supset Q_i \text{ for all } i}}x_{Q_{m+2}}k_{Q_1, Q_{m+2}} \alpha_{Q_{m+2}}\right)^\gamma \beta_{Q_1} \nonumber \\
        & \lesssim_p \sum_{\substack{Q_1, ..., Q_{m+1} \in \mc{D}: \\ Q_1 \subset \cdots \subset Q_{m+1}}}\left(\prod_{j = 2}^{m+1}x_{Q_j}k_{Q_1,Q_j}\alpha_{Q_j}\right)\left(\sum_{\substack{Q_{m+2} \in \mc{D}: \\ Q_{m+2} \supset Q_2}}x_{Q_{m+2}}k_{Q_1, Q_{m+2}} \alpha_{Q_{m+2}}\right)^\gamma \beta_{Q_1} \nonumber \\
        & \leq A \sum_{Q_1, ..., Q_{m+1} \in \mc{D}}k_{Q_1, Q_2}^{p'} \left(\prod_{j = 2}^{m+1}x_{Q_j}k_{Q_2,Q_j}\alpha_{Q_j}\right)\left(\sum_{\substack{Q_{m+2} \in \mc{D}: \\ Q_{m+2} \supset Q_2}}x_{Q_{m+2}}k_{Q_2, Q_{m+2}} \alpha_{Q_{m+1}}\right)^\gamma \beta_{Q_1} \label{eqn:big_sum_3}
    \end{align}
    where the second inequality is justified in a similar way as above: for a summand in the sum in the first line to be nonzero, we must have $Q_1 \subset Q_j$ for all $j$ by definition of $k$. Then by reordering the cubes, we can assume that $Q_1 \subset \cdots \subset Q_{m+1}$. By \cref{eqn:carleson_condition_adj_2}, we have that 
    \begin{equation*}
        \sum_{Q_1 \in \mc{D}}k_{Q_1, Q_2}^{p'} \beta_{Q_1} = \sum_{Q_1 \in \mc{D}(Q_2)}\ang{w_{Q_2}}_{Q_1} \mu(Q_1)\alpha_{Q_1} \leq \mu(Q_2) C_2.
    \end{equation*}
    Substituting this estimate into \cref{eqn:big_sum_3} gives
    \begin{align*}
        \|z\|_{l^{p'}(\beta)}^{p'} & \lesssim_{p} AC_2 \sum_{Q_2, ..., Q_{m+1} \in \mc{D}} \left(\prod_{j = 2}^{m+1}x_{Q_j}k_{Q_2,Q_j}\alpha_{Q_j}\right)\left(\sum_{\substack{Q_{m+2} \in \mc{D}: \\ Q_{m+2} \supset Q_2}}x_{Q_{m+2}}k_{Q_2, Q_{m+2}} \alpha_{Q_{m+1}}\right)^\gamma \mu(Q_2) \\
        & = AC_2\sum_{Q_2 \in \mc{D}}x_{Q_2}\left(\sum_{Q_3 \in \mc{D}}x_{Q_3}k_{Q_2,Q_3}\alpha_{Q_{3}}\right)^{p'-1}\beta_{Q_2} \\
        & = AC_2I.
    \end{align*}

Therefore by \cref{eqn:holders},
    \[I \leq \|x\|_{l^{p'}(\beta)} \|z\|_{l^{p'}(\beta)}^{p' - 1} \lesssim_{p} \|x\|_{l^{p'}(\beta)} (AC_2I)^{(p' - 1) / p'} = \|x\|_{l^{p'}(\beta)} (AC_2I)^{1/p}.\]
    We conclude that 
    \[\|H^*x\|_{L^{p'}} \lesssim_{p} A^{1/(pp')} I^{1/p'}\lesssim_{p} A^{1/(pp') + 1/p}C_2^{1/p} \|x\|_{l^{p'}(\beta)}.\]
    By duality, see that 
    \[\|Hf\|_{\ell^p(\beta)}^p \lesssim_{p} A^{1 + 1/p'}C_2\|f\|_{L^p}^p\]
    which proves \cref{eqn:operator_bound} and concludes the proof of the theorem.
\end{proof}

We are now ready to give the proof of \cref{lem:square_function_estimate}.

\begin{proof}[Proof of \cref{lem:square_function_estimate}]
    We need to show that 
    \begin{equation}\label{eqn:square_function_bound}
        \sum_{Q \in \mc{S}}\ang{\bv{\mc{W}_Q^{-1}W^{1/p}}\,|g|}_Q^{p'}\mu(Q) \lesssim_{d,\eta} [W]_{A_{p,\infty}^{sc}}\|g\|_{L^{p'}}^{p'}
    \end{equation}
    where $\mc{S}$ is an $\eta$-sparse family and $W$ satisfies the $A_{p,\infty}^{sc}$ condition. \par
    For each $Q \in \mc{D}$, define 
    \[w_Q := \bv{\mc{W}_{Q}^{-1}W^{1/p}}^p\1_Q, \quad \alpha_Q := \begin{cases}\mu(Q), & Q \in \mc{S} \\ 0, & Q \not \in \mc{S}\end{cases}.\]
    We will apply \cref{thm:carleson_embedding_general} with this choice of weights and coefficients. Observe that
    \[\ang{w_Q}_{Q} \sim_d 1\]
    by definition of the reducing operator. We verify \cref{eqn:weight_condition}. Indeed, for $P, Q \in \mc{D}$ with $Q \subset P$, we have that on $Q$,
    \begin{align*}
        w_P & = \bv{\mc{W}_{P}^{-1}W^{1/p}}^p \\
        & \leq \bv{\mc{W}_{P}^{-1}\mc{W}_Q}^p\bv{\mc{W}_{Q}^{-1}W^{1/p}}^p \\
        & \sim_d \left(\dashint_{Q}\bv{\mc{W}_{P}^{-1}W^{1/p}}^p \, d\mu\right)\bv{\mc{W}_{Q}^{-1}W^{1/p}}^p \\
        & = \ang{w_P}_{Q}w_Q.
    \end{align*}
    Therefore \cref{eqn:weight_condition} is verified with $A \lesssim_d 1$. We now check \cref{eqn:carleson_condition_2}. Note that $w_P \in A_{\infty}$ with 
    \[[w_P]_{A_\infty} \lesssim_{d,p} [W]_{A_{p, \infty}^{sc}}.\]
    Indeed, if $\{e_1, ..., e_{d}\}$ is an orthonormal basis of $\R^d$, then
    \[w_P \sim_{d,p} \sum_{j}\bv{W^{1/p}\mc{W}_{P}^{-1}e_j}^p \1_{P}.\]
    Each of the weights $w_P^j = \bv{W^{1/p}\mc{W}_{P}^{-1}e_j}^p \1_{P}$ are scalar $A_\infty$ weights with $[w_P^j]_{A_\infty} \leq [W]_{A_{p, \infty}^{sc}}$. Using the fact that 
    \[[v_1 + \cdots + v_m]_{A_\infty} \leq \max_j \, [v_j]_{A_\infty}\]
    for $A_\infty$ weights $\{v_j\}$ (see \cite{NPTV}), we have that
    \[[w_P]_{A_\infty} \lesssim_d \max_j \, [w_P^j]_{A_\infty} \leq [W]_{A_{p, \infty}^{sc}}\]
    Now, since $\mc{S}$ is $\eta$-sparse, for each cube $Q$, there is some subset $E_{Q} \subset Q$ so that $\mu(E_{Q}) \geq \eta\mu(Q)$ and the collection $\{E_{Q}\}_{Q \in \mc{D}}$ is pairwise disjoint. It follows that

    \begin{align*}
        \frac{1}{\mu(P)}\sum_{Q \in \mc{D}(P)}\ang{w_P}_{Q}\ang{w_Q}_Q^{p'-1}\alpha_Q & \lesssim_{d,p} \frac{1}{\mu(P)}\sum_{Q \in \mc{S}(P)}\ang{w_P}_{Q}\mu(Q) \\
        & \lesssim_\eta \frac{1}{\mu(P)}\sum_{Q \in \mc{S}(P)}\ang{w_P}_{Q}\mu(E_Q) \\
        & \lesssim \frac{1}{\mu(P)}\sum_{Q \in \mc{S}(P)}\int_{E_{Q}}M_{P}^\mc{D}(w_P) \, d\mu \\
        & \leq \frac{1}{\mu(P)} \int_{P}M_{P}^{\mc{D}}(w_P) \, d\mu \\
        & \leq [w_P]_{A_\infty} \ang{w_P}_P \\
        & \lesssim_{d,p} [W]_{A_{p, \infty}^{sc}},
    \end{align*}
    where $\mc{S}(P)= \mc{D}(P) \cap \mc{S}.$ Therefore 
    \[\frac{1}{\mu(P)}\sum_{Q \in \mc{D}(P)}\ang{w_P}_{Q}\ang{w_Q}_Q^{p'-1}\alpha_Q  \lesssim_{d, p, \eta} [W]_{A_{p, \infty}^{sc}}\]
    and so \cref{eqn:square_function_bound} follows from \cref{thm:carleson_embedding_general}.
\end{proof}

We are now ready to prove \cref{MatrixWeightedThm}.

\begin{proof}[Proof of \cref{MatrixWeightedThm}]
    We want to show that 
    \[\mc{A}_{\mc{S}}(V^{1/p'}f, W^{1/p}g) \lesssim [W]_{A_{p, \infty}^{sc}}^{1/p'} [W, V]_{A_p}^{1/p}[V]_{A_{p', \infty}^{sc}}^{1/p}\]
    where $\|f\|_{L^p(\R^n; \R^d)} = \|g\|_{L^{p'}(\R^n; \R^d)} = 1$. By the computations at the start of this section and \cref{lem:square_function_estimate}, we have 
\[\mc{A}_{\mc{S}}(V^{1/p'}f, W^{1/p}g) \lesssim_d [W, V]_{A_p}^{1/p} \|S^V_{p} f\|_{L^{p}} \|S^W_{p'}g\|_{L^{p'}} \lesssim_{d,p} [W]_{A_{p, \infty}^{sc}}^{1/p'} [W, V]_{A_p}^{1/p}[V]_{A_{p', \infty}^{sc}}^{1/p}\]
as wanted. In the one weight case, we set $V = W^{-p'/p}$. Then 
\[[W, V]_{A_p} = [W]_{A_p}\]
so by \cref{prop:weight_duality},
\begin{align*}
    \mc{A}_S(W^{-1/p}f, W^{1/p}g) & \lesssim_d [W]_{A_{p, \infty}^{sc}}^{1/p'} [W]_{A_p}^{1/p}[W^{-p'/p}]_{A_{p', \infty}^{sc}}^{1/p} \\
    & \leq [W]_{A_p}^{1/p'} [W]_{A_p}^{1/p}[W^{-p'/p}]_{A_{p'}}^{1/p} \\
    & \lesssim [W]_{A_p}^{1/p'} [W]_{A_p}^{1/p}[W]_{A_p}^{p'/p^2} \\
    & = [W]_{A_p}^{1 + \frac{1}{p-1} - \frac{1}{p}}.
\end{align*}
Now suppose $T$ is an operator that admits vector dual sparse domination
\[|\ang{Tf, g}| \lesssim \mc{A}_{\mc{S}}(f, g)\]
where $\mc{S}$ is a sparse family with sparse parameter independent of $f$ and $g$. By a simple substitution, 
\[\|T\|_{L^p(W) \to L^p(W)} = \|M_{W^{1/p}}TM_{W^{-1/p}}\|_{L^p \to L^p}.\]
We therefore estimate  
\begin{align*}
    \|M_{W^{1/p}}TM_{W^{-1/p}}f\|_{L^p} & = \|W^{1/p}T(W^{-1/p}f)\|_{L^p}  \\
    & = \sup_{\|g\|_{L^{p'}} = 1}|\ang{W^{1/p}TW^{-1/p}f, g}| \\
    & = \sup_{\|g\|_{L^{p'}} = 1}|\ang{TW^{-1/p}f, W^{1/p}g}| \\
    & \lesssim \sup_{\|g\|_{L^{p'}} = 1}\mc{A}_{\mc{S}}(W^{-1/p}f, W^{1/p}g) \\
    & \lesssim_d [W]_{A_p}^{1 + \frac{1}{p-1} - \frac{1}{p}}\|f\|_{L^p}
\end{align*}
as wanted.
\end{proof}

\begin{proof}[Proof of \cref{cor:martingale_multiplier}] 
The result holds for $L^1$ normalized shifts combining \cref{T:thm A 2nd part} and \cref{MatrixWeightedThm}. For martingale multipliers, by \cref{lem:pointwise_to_form} and \cref{MatrixWeightedThm}, it suffices to show that for each compactly supported $f \in L^1(\R^n;\R^d)$, there is some $(1/2)-$sparse family $\mc{S} \subset \mc{D}$ so that 
\[T_\sigma f(x) \in C\sum_{Q \in \mc{S}}\dang{f}_Q\1_Q(x)\]
almost everywhere on $Q_0$. The proof of the convex body sparse domination is similar to the proof of the corresponding scalar result (see \cite{Lacey_2017}) so we omit some of the details. Suppose $f$ is supported on a dyadic cube $Q_0$. Let $\mc{B}(Q_0)$ be the maximal cubes $Q \in \mc{D}$ so that  
\[\sum_{Q' \in \mc{D}: Q' \supsetneq Q}\sigma_{Q'} \Delta_{Q'}f(x) \not \in C \dang{f}_{Q_0} \quad \text{or} \quad \ang{f}_Q \not \in C \dang{f}_{Q_0}.\]
Using the weak $(1,1)$ inequality for martingale multipliers and the maximal function, one can see that if $C$ is large enough, then 
\[\sum_{Q \in \mc{D}(Q_0)}\mu(Q) \leq \frac{1}{2}\mu(Q_0).\]
We decompose 
\begin{align*}
    T_\sigma f(x) & = T_\sigma f(x) \1_{Q_0 \setminus \bigcup_{Q \in \mc{B}(Q_0)}}(x)  + \sum_{Q \in \mc{B}(Q_0)}T_\sigma f(x) \1_Q(x)\\
    & = T_\sigma f(x) \1_{Q_0 \setminus \bigcup_{Q \in \mc{B}(Q_0)}}(x) \\ 
    & + \sum_{Q \in \mc{B}(Q_0)}\left(\sum_{Q' \in \mc{D}: Q' \supsetneq \widehat{Q}}\sigma_{Q'} \Delta_{Q'}f(x)\right) \1_Q(x) \\ & + \sum_{Q \in \mc{B}(Q_0)}\sigma_{\widehat{Q}}\Delta_{\widehat{Q}}f(x) \1_Q(x) \\ & + \sum_{Q \in \mc{B}(Q_0)}\left(\sum_{Q' \in \mc{D}: Q' \subset Q}\sigma_{Q'} \Delta_{Q'}f(x)\right) \1_Q(x).
\end{align*}
The first two terms are contained in $C \dang{f}_{Q}$ by the stopping condition and the third term is contained in $C\dang{f}_{Q_0} + C \sum_{Q \in \mc{D}(Q_0)}\dang{f}_{Q}\1_Q(x)$ by support considerations and since $\ang{f}_{\widehat{Q}} \in C \dang{f}_{Q_0}$. By performing this argument inductively on the sub-cubes of all $Q \in \mc{B}$ and forming the sparse family $\mc{S}$ with such sub-cubes, we conclude the proof of the corollary.
\end{proof}

\section{Modified Weight Classes for Haar Shifts} \label{ModifiedClass}
\subsection{Necessity of the modified weight condition and stabilization} Let $N \in \mathbb{N}$, $1<p<\infty$, and recall the complexity-dependent matrix weight class $A_p^N$ introduced in Section \ref{section 2}.
It is clear that, in general, $[W]_{A_p}\leq [W]_{A_p^N}$. The main goal of this subsection is to show that the $A_p^N$ condition is the right condition to study weighted estimates for Haar shifts.
\begin{theorem} \label{thm:necessity}
Let $n=1$, $W$ be a $d \times d$ matrix weight, and $1<p< \infty$. Suppose that all Haar shifts of complexity at most $N \in \mathbb{N}$ are bounded uniformly on $L^p(W)$; that is, 
$$\max_{\substack{ s,t \in \mathbb N: \\ s+t \leq N}} \sup_{T \in HS(s,t, \mathcal{H}_{\D})} \|T\|_{L^p(W)} \leq C.  $$
Then $W \in A_p^N.$
    
\end{theorem}

   \begin{remark}
The restriction $n=1$ in Theorem \ref{thm:necessity} is somewhat natural, because Haar shifts are canonical objects on the real line, while in higher dimensions there are, in general, many different choices of Haar systems. When $n \geq 2$ and the Haar system $\mathcal{H}_{\D}$ satisfies the ``non-degeneracy'' condition

$$ |\alpha_R| \gtrsim \frac{\sqrt{m(\widehat{R})}}{\mu(R)}, \quad R \in \D,$$
then Theorem \ref{thm:necessity} holds with the same proof, provided we interpret ``complexity $(0,0)$'' Haar shifts in this context as martingale multipliers as defined below rather than Haar multipliers (these notions coincide when $n=1$).

\end{remark}

Theorem \ref{thm:necessity} is a consequence of the following proposition and lemma. Proposition \ref{MartingaletoAp} is valid for all $n \in \N.$

\begin{proposition} \label{MartingaletoAp}
Let $W$ be a $d \times d$ matrix weight and $1<p<\ \infty$.  Suppose that all martingale multipliers are uniformly bounded on $L^p(W)$; that is, 
 $$  \sup_{\sigma} \|T_{\sigma}\|_{L^p(W)} \leq C,$$
 where $$ T_{\sigma}f=  \sum_{Q \in \mc{D}}\sigma_Q \Delta_Q f$$
 and $\sigma_Q=\pm 1$ for all $Q$ (the supremum is taken over all such sequences of signs). Then we must have $W \in A_p$ and $[W]_{A_p} \lesssim C^p.$

 \begin{proof}
This is a consequence of the theory of unconditional bases in Banach spaces, as well as Proposition \ref{ExpOp}. In particular, the argument is essentially the same as the one given in the scalar case with $p=2$ in \cite[Proposition 2.3]{ThieleTreilVolberg} (see also Theorem 2.2 in the same reference).

 \end{proof}
\end{proposition}

The ``new conditions'' that we must check for the $A_p^N$ condition are taken care of in the lemma below. Recall that for each pair of dyadic cubes $I,J$, there exist reducing operators ($d \times d$ matrices) $\mathcal{W}_{I}$ and $\mathcal{V}_{J}$ satisfying for each vector $e \in \R^d$:

$$ |\mathcal{W}_{I} e |^p \sim \dashint_{I} |W^{1/p}(x)e |^p \, d\mu(x)  , \quad  |\mathcal{V}_{J} e |^{p'} \sim \dashint_{J} |W^{-1/p}(x)e |^{p'} \, d\mu(x) $$
It is then possible to deduce (see Section \ref{section 2} and \cite{cruzuribe2017}) that
\begin{equation}\label{eqn:reducing_operator_modified}
    [W]_{A_p^N} \sim \sup_{\substack{I,J \in \mc{D} \\ 0 \leq \text{dist}_{\mc{D}}(I, J) \leq N + 2}}[c_p^b(I, J)]   \bv{\mathcal{W}_{I} \mathcal{V}_{J}}^p,
\end{equation}
and also
$$ [W]_{A_p} \sim \sup_{Q \in{D}} \bv {\mathcal{W}_{Q} \mathcal{V}_{Q}}^p.$$

\begin{lemma}
Let $n=1$, $W$ be a matrix weight, let $N \in \mathbb N$ and suppose

$$\max_{\substack{ s,t \in \mathbb N: \\ s+t \leq N}} \sup_{T \in HS(s,t)} \|T\|_{L^p(W)} \leq C.  $$

Let $J, K \in \D$ be such that $\text{dist}_{\D}(J,K) \leq N+2$. Then

$$ c_p(J,K) \bv{ \mathcal{W}_{J} \mathcal{V}_{K}}^p  \lesssim C^{3p}.$$

\end{lemma}
\begin{proof}

First, taking $s=t=0$, notice that Proposition \ref{MartingaletoAp} immediately implies, for any interval $Q \in \mc{D}$, there holds $$\bv{ \mathcal{W}_{Q} \mathcal{V}_{Q} } \lesssim C.$$

Let $L$ be the minimal common dyadic ancestor of $J,K$. Let us first consider the case when $J,K$ are both strictly contained in $L$. Let $s,t \in \N $ be such that $K \in \D_{s+1}(L),$ $J \in \D_{t+1}(L)$, and notice necessarily $s+t \leq N.$ Let $e$ be a unit vector in $\R^d$ chosen so that

$$ \left | \mathcal{V}_K e \cdot  \mathcal{W}_J e \right| \geq \frac{1}{2} \bv  {\mathcal{V}_K \mathcal{W}_J}.      $$

Then choose test vector-valued functions

$$ f_1(x)= \frac{\mathbf{1}_K}{\mu(K)^{1/p}} (\mathcal{V}_K e), \quad f_2(x)= \frac{\mathbf{1}_J}{\mu(J)^{1/p'}} (W^{-1}(x) \mathcal{W}_Je).$$

We now compute the $L^p(W)$ norm of $f_1$ and the $L^{p'}(W)$ norm of $f_2$:

\begin{align*}
\|f_1\|_{L^p(W)}^p & = \frac{1}{\mu(K)} \int_{K} |W^{1/p}(x) \mathcal{V}_K e|^p \, d \mu(x) \\
& \sim | \mathcal{W}_K \mathcal{V}_K e |^p \\
& \leq \bv {\mathcal{W}_K \mathcal{V}_K} ^p \\
& \lesssim C^p
\end{align*}

and similarly, $\|f_2\|_{L^{p'}(W)} \lesssim C.$

Now, let $\{\alpha_{R,S}^{Q}\}$ be the sequence of scalars satisfying $\alpha_{\hat{K},\hat{J}}^L=1$ and $\alpha_{R,S}^{Q}=0$ otherwise, and let $T$ be the associated Haar shift operator of rank one, i.e.

$$ Tf= \langle f, h_{\hat{K}} \rangle h_{\hat{J}}.$$

With this in mind and considering our choices of $f_1,f_2$, we have by duality

\begin{align*}
C^3 & \geq \left | \int_{\R} T f_1 \cdot W f_2 \, d\mu \right | \\
& = \left| \langle f_1, h_{\hat{K}} \rangle \cdot \langle W f_2, h_{\hat{J}} \rangle      \right | \\
& = \frac{\sqrt{m(\widehat{K})m(\widehat{J}})}{\mu(K) \mu(J)} \left |\int_{K} \int_{J} f_1 (y) \cdot W(x) f_2 (x)   \,d \mu(x)\, d \mu(y) \right|\\
& \sim \frac{\sqrt{m(K)m(J)}}{\mu(K)^{1/p} \mu(J)^{1/p'}} \left | \frac{1}{\mu(J) \mu (K)} \int_{K} \int_{J} \mathcal{V}_K e \cdot \mathcal{W}_J e \, d \mu(y) d \mu(x)    \right | \\
& \sim c_p^b(J,K)^{1/p} |\mathcal{V}_K e \cdot \mathcal{W}_J e| \\
& \gtrsim c_p^b(J,K)^{1/p} \bv{\mathcal{W}_J  \mathcal{V}_K},
\end{align*}
which implies the desired estimate.

It remains to handle the case when $1 \leq \operatorname{dist}_{\D}(J,K) \leq N+2$ but either $L=J$ or $L=K.$ Without loss of generality, let us assume $L=K$ and thus $J \subsetneq K. $ Then if $e$ is any unit vector, it is enough to show
$$
c_p(J,K) |\mathcal{W}_{J} \mathcal{V}_{K} e|^p  \lesssim C^p.
$$
To see this, note that

\begin{align*}
 c_p(J,K) |\mathcal{W}_{J} \mathcal{V}_{K} e|^p  & \sim \frac{m(J)^p}{\mu(K) \mu(J)^{p-1} }\frac{1}{\mu(J)} \int_{J} |W^{1/p}(x)(\mathcal{V}_k e)| \, d \mu(x)\\
 & \leq \frac{m(J)^p}{ \mu(J)^{p} }\frac{1}{\mu(K)} \int_{K} |W^{1/p}(x)(\mathcal{V}_k e)| \, d \mu(x) \\
 & \lesssim |\mathcal{W}_{K} \mathcal{V}_{K} e|^p \lesssim C^p,
\end{align*}
as required.

\end{proof}

Although we define complexity-dependent weight characteristics $[W]_{A_p^N}$, the weight \textit{classes} are the same independently of the complexity. They are all unified under the following condition.
\begin{definition}
    Let $1<p<\infty$. We say that $W\in A_p^b$ if
    \begin{equation*}
        \sup_{\substack{Q,R\in\mathcal D\\ R\in\mathrm{ch}(\hat Q)\cup\mathrm{ch}\left(Q^{(2)}\right)\\\text{or } Q\in\textrm{ch}\left(R^{(2)}\right)}}c_p^b(Q,R)\dashint_Q\left(\dashint_R|W(x)^{1/p}W(y)^{-1/p}|^{p'}d\mu(y)\right)^{p/p'}d\mu(x)<\infty,
    \end{equation*}
    where $Q^{(1)}=\hat Q$ and $Q^{(j)}=\widehat{Q^{(j-1)}}$ for $j\geq 2$.
\end{definition}
\begin{proposition}\label{prop:stability}
    For $1<p<\infty$ and $N\in\N$, we have
    \begin{equation*}
    [W]_{A_p^b}\leq[W]_{A_p^N}\lesssim\left([W]_{A_p^b}\right)^{2^{N-1}}.
    \end{equation*}
    In particular, $A_p^N=A_p^M$ for all $N,M\in\N$.
\end{proposition}
\begin{remark}
    This was already observed in \cite{CPW} when $d=n=1$. Our definition of $A_p^b$ is stronger, as they needed not consider the case in which $R$ is a sibling of $Q$. We include this configuration in the supremum to deal with the situation in which $Q$ and $R$ are siblings but neither $\mu(Q)$ nor $\mu(R)$ are comparable to $\mu(\hat Q)$. This can only happen when $n>1$, and so for $n=1$ we can replace the $A_p^b$ condition by
    \begin{equation*}
        \sup_{\substack{Q,R\in\mathcal D\\ R\in\lbrace Q\rbrace\cup\mathrm{ch}\left(Q^{(2)}\right)\\\text{or } Q\in\textrm{ch}\left(R^{(2)}\right)}}c_p^b(Q,R)\dashint_Q\left(\dashint_R|W(x)^{1/p}W(y)^{-1/p}|^{p'}d\mu(y)\right)^{p/p'}d\mu(x)<\infty.
    \end{equation*}
\end{remark}
\begin{proof}[Proof of \cref{prop:stability}]
    It is always the case that $[W]_{A_p^b}\leq[W]_{A_p^N}$. To prove the second inequality, we proceed by induction on $N$. When $N=1$, we only need to consider the cases $Q\subset R$ and $R\subset Q$. For $Q\subset R$,
    \begin{align*}
        &\frac{m(Q)^{p/2}m(R)^{p/2}}{\mu(Q)^{p-1}\mu(R)}\dashint_Q\left(\dashint_R|W(x)^{1/p}W(y)^{-1/p}|^{p'}d\mu(y)\right)^{p/p'}d\mu(x)\\&\lesssim \frac{m(Q)^{p/2}m(R)^{p/2}}{\mu(Q)^p}\dashint_R\left(\dashint_R|W(x)^{1/p}W(y)^{-1/p}|^{p'}d\mu(y)\right)^{p/p'}d\mu(x)\\&\lesssim\dashint_R\left(\dashint_R|W(x)^{1/p}W(y)^{-1/p}|^{p'}d\mu(y)\right)^{p/p'}d\mu(x)\leq[W]_{A_p^b}.
    \end{align*}
    An analogous argument holds when $R\subset Q$. Now fix $N\in\N$ and assume $[W]_{A_j^N}\lesssim([W]_{A_p^b})^{2^{N-1}}$. We shall show $[W]_{A_j^{N+1}}\lesssim[W]_{A_j^N}^2$. We only need to consider the cases in which $\mathrm{dist}(Q,R)=N+3$. Furthermore, if $Q\cap R\neq\emptyset$, the argument for the base case works, so we assume $Q\cap R=\emptyset$. Let $K$ be the minimal common ancestor of $Q$ and $R$, so that $Q^{(s)}=R^{(t)}=K$ and $s+t=N+3$. In particular, $Q^{(2)}\subset K$ or $R^{(2)}\subset K$. Without loss of generality, suppose the former. Take $Q^*\in\mathrm{ch}(Q^{(2)})$ such that
    \begin{equation*}
        \mu(Q^*)=\min\lbrace \mu(S):S\in\mathrm{ch}(Q^{(2)})\rbrace\sim m(Q^{(2)})\sim m(Q^*).
    \end{equation*}
    From that, it holds that $c_p^b(Q,R)\sim c_p^b(Q,Q^{*})c_p^b(Q^{*},R)$.
    Now, we use the characterization of $[W]_{A_p^{N+1}}$ via reducing operators. In general, we have that $\|\mathcal V_{Q^*}\mathcal W_{Q^*}v\|\geq\|v\|$ for all $v\in\R^d$ (see \cite[Proposition 1.1]{Goldberg2003}). By Douglas' lemma, there exists a matrix $T$ such that $\mathcal W_{Q^*}^{-1}=T\mathcal V_{\hat Q^*}$, and we have that $\bv{T}\leq 1$. In particular, $\|T^{-1}v\|\geq \|v\|$ for all $v\in\R^d$. We have the estimate
\begin{align*}
    c_p^b(Q,R)\bv{\mathcal W_Q\mathcal V_R}^p&=c_p^b(Q,R)\bv{\mathcal W_Q\mathcal V_{Q^*}\mathcal V_{Q^*}^{-1}\mathcal V_R}^p\lesssim c_p^b(Q, Q^*)\bv{\mathcal W_Q\mathcal V_{Q^*}}^pc_p^b(Q^*,R)\bv{\mathcal V_{Q^*}^{-1}\mathcal V_R}^p\\&\lesssim c_p^b(Q, Q^*)\bv{\mathcal W_Q\mathcal V_{Q^*}}^pc_p^b(Q^*, R)\bv{\mathcal V_R \mathcal V_{Q^*}^{-1}T^{-1}}^p\\&\sim c_p^b(Q, Q^*)\bv{\mathcal W_Q\mathcal V_{Q^*}}^pc_p^b(Q^*,R)\bv{\mathcal W_{Q^*}\mathcal V_R}^p\leq [W]_{A_p^N}^2,
\end{align*}
as $\mathrm{dist}(Q,Q^*)\leq 3\leq N+2$ and $\mathrm{dist}(Q^*,R)\leq N+2$. This completes the induction.
\end{proof}

\subsection{Weighted estimates for Haar shifts}

\begin{proof} [Proof of Corollary \ref{cor:WeightedHaarShift}]
   By \cref{thm: sparseforvectorhaar} and \cref{lem:pointwise_to_form}, we see that for any $f \in L^p$ and $g \in L^{p'}$, we have that 
   \[|\ang{Tf, g}| \lesssim \mc{A}_{\mc{S}}(f, g) + \mc{A}_{\mc{S}}^N(f, g)\]
   where $\mc{S}$ is an $\eta$-sparse family with $\eta$ is independent of $f$ and $g$. Similarly to the proof of \cref{MatrixWeightedThm}, to conclude the proof of the corollary, we need to show that for all $f \in L^p$ and $g \in L^{p'}$, we have that
   \[\mc{A}_{\mc{S}}(V^{1/p'}f, W^{1/p}g) + \mc{A}^N_{\mc{S}}(V^{1/p'}f, W^{1/p}g) \lesssim [W]_{A_p}^{1 + \frac{1}{p-1} - \frac{2}{p}}[W]_{A_p^N}^{\frac{1}{p}} \|f\|_{L^p}\|g\|_{L^{p'}}\]
   where $V = W^{-p'/p}$. The bound of the first term follows by \cref{MatrixWeightedThm} and the fact that $[W]_{A_p} \leq [W]_{A_p^N}$ for all $N$. Therefore it remains to bound the $\mc{A}_{\mc{S}}^N$ term:
   \[\sum_{\substack{J,K \in \mc{S} \\ \text{dist}_{\mc{D}}(J, K) \leq N + 2 }}\dashint_{J}\dashint_{K}|V^{1/p'}(x)f(x) \cdot W^{1/p}(y) g(y)| \, d\mu(x)d\mu(y) \sqrt{m(J)}\sqrt{m(K)}.\]
   Indeed, by following the discussion at the start of \cref{section 4}, and by applying \cref{eqn:reducing_operator_modified} we can see that
   \[|V^{1/p'}(x)f(x) \cdot W^{1/p}(y)g(y)|\sqrt{m(J)}\sqrt{m(K)}\] 
   \[\lesssim_d [W]_{A_{p}^N}^{1/p}|\mc{V}^{-1}_K V^{1/p'}(x)f(x)| \, |\mc{W}_J^{-1}W^{1/p}(y)g(y)|\mu(J)^{1/p'}\mu(K)^{1/p}.\]
   Therefore 
   \begin{align*}
       & \mc{A}^N_{\mc{S}}(V^{1/p'}f, W^{1/p}g) \\ & \lesssim_d [W]_{A_{p}^N}^{1/p} \sum_{\substack{J,K \in \mc{S} \\ \text{dist}_{\mc{D}}(J, K) \leq N + 2 }}\dashint_{J}\dashint_{K}|\mc{V}^{-1}_K V^{1/p'}(x)f(x)| \, |\mc{W}_J^{-1}W^{1/p}(y)g(y)| \, d\mu(x)d\mu(y) \mu(J)^{1/p'}\mu(K)^{1/p} \\
       & \lesssim_{N} [W]_{A_{p}^N}^{1/p} \left(\sum_{K \in \mc{S}}\ang{|\mc{V}^{-1}_K V^{1/p'}f|}_K^p \mu(K)\right)^{1/p}\left(\sum_{J \in \mc{S}}\ang{|\mc{W}^{-1}_J V^{1/p'}f|}_J^{p'} \mu(J)\right)^{1/p'} \\
       & \lesssim_{n, p, \eta} [W]_{A_p^N}^{1/p}[V]_{A_{p'}}^{1/p}[W]_{A_{p}}^{1/p'}\|f\|_{L^p}\|g\|_{L^{p'}} \\
       & \lesssim_d [W]_{A_p}^{1 + \frac{1}{p-1} - \frac{2}{p}}[W]_{A_p^N}^{\frac{1}{p}} \|f\|_{L^p}\|g\|_{L^{p'}}
   \end{align*}
   where the second last inequality follows from \cref{lem:square_function_estimate} and the last inequality follows from \cref{prop:weight_duality}.

\end{proof}

\subsection{Two-weight bumps}
Recall a Young function $\Phi:[0,\infty) \rightarrow [0,\infty)$ is a convex, continuous increasing function which satisfies $\Phi(0)=0$. We also assume $\lim_{t \rightarrow \infty}\frac{\Phi(t)}{t}=\infty.$ Given a Young function $\Phi$, its dual or complementary Young function is denoted by $\overline{\Phi}$ and is defined via 

$$ \overline{\Phi}(t):= \sup_{s>0} st- \Phi(s). $$
We have the generalized H\"{o}lder inequality 

$$\frac{1}{\mu(Q)}\int_{Q} |fg| \, d \mu \leq 2 \|f\|_{\Phi, Q} \|g\|_{\overline{\Phi},Q},$$

where
$$ \|f\|_{\Phi, Q}:= \inf\left\{ \lambda>0: \frac{1}{\mu(Q)}\int_{Q} \Phi\left(\frac{|f(x)|}{\lambda}\right) \, d \mu(x) \leq 1  \right \}$$
is the usual local Orlicz norm on the cube $Q$.

If $1<p<\infty$, we say that $\Phi$ satisfies the $B_p$ condition if

$$\int_{1}^{\infty} \frac{\Phi(t)}{t^p} \frac{dt}{t}<\infty.$$
Note this integral will remain finite if we replace the lower limit $1$ by any $c>0.$ This condition will be relevant in maximal function bounds and two weight bump estimates.
Define the dyadic (scalar) Orlicz maximal operator

$$M_{\Phi}^{\D}f(x):= \sup_{Q \in \D} \|f\|_{\Phi, Q} \1_{Q}(x),$$

We need a technical result concerning this Orlicz maximal function in order to prove the two weight $L^p$ estimates with bump conditions on the weight.
The following lemma is due to Perez \cite{Perez1995} in the homogeneous case, but for the dyadic variant of the Orlicz maximal function it actually works for any Borel measure $\mu.$

\begin{lemma} \label{MaxOrliczBound}
For each $1<p<\infty$ and $\Phi \in B_p$, there exists $C_{p,\Phi}>0$ so that

$$ \int_{\R^n} |M_{\Phi}^{\D} f|^p \, d\mu \leq C_{p, \Phi} \int_{\R^n}|f|^p \, d\mu, \quad f \in L^p(\mu).$$
That is, $M_{\Phi}^{\D}$ is bounded on $L^p(\mu)$.

\begin{proof}
We can assume by density $f$ is bounded and compactly supported. Let $c>0$ be large enough so that $\Phi(c^{-1}) \leq \frac{1}{2}.$ We claim the distributional inequality 

$$ \mu(\{ |M_{\Phi}^{\D}f|> \lambda \}) \leq 2 \int_{\{|f|>\frac{\lambda}{c}\}} \Phi\left( \frac{|f|}{\lambda} \right) \, d\mu.$$

To see this is true, write  $\{ |M_{\Phi}^{\D}f|> \lambda \}$ (which we can assume has finite measure by assumptions on $\mu$ and the compact support of $f$) as a union of maximal, disjoint dyadic cubes $Q$ satisfying $\|f\|_{\Phi, Q}>\lambda$, which easily gives
\begin{align*}
 \mu(\{ |M_{\Phi}^{\D}f|> \lambda \}) & \leq \int_{\{ |M_{\Phi}^{\D}f|> \lambda \}} \Phi\left( \frac{|f|}{\lambda} \right) \, d \mu \\
 & \leq  \int_{\{|f|>\frac{\lambda}{c}\}} \Phi\left( \frac{|f|}{\lambda} \right) \, d\mu + \int_{\substack{\{ |M_{\Phi}^{\D}f|> \lambda \} \cap \\ \{|f|\leq \frac{\lambda}{c}\} }} \Phi\left( \frac{|f|}{\lambda} \right) \, d \mu \\
 & \leq \int_{\{|f|>\frac{\lambda}{c}\}} \Phi\left( \frac{|f|}{\lambda} \right) \, d\mu + \int_{\{|M_{\Phi}^{\D}f|> \lambda \}} \Phi\left( \frac{1}{c} \right) \, d \mu \\
 & \leq \int_{\{|f|>\frac{\lambda}{c}\}} \Phi\left( \frac{|f|}{\lambda} \right) \, d\mu + \frac{1}{2} \mu(\{ |M_{\Phi}^{\D}f|> \lambda \}),
\end{align*}
which proves the claim after rearrangement.

Using this claim, we then estimate via the change of variable $t=\frac{|f(x)|}{\lambda}$:

\begin{align*}
\|M_{\Phi}^{\D}\|_{L^p(\mu)}^p & = p \int_{0}^{\infty} \lambda^{p-1}  \mu(\{ |M_{\Phi}^{\D}f|> \lambda \}) \, d \lambda   \\
& \leq 2p \int_{0}^{\infty} \lambda^{p-1}\int_{\{|f|>\frac{\lambda}{c}\}} \Phi\left( \frac{|f(x)|}{\lambda} \right) \, d\mu(x) \, d \lambda\\
& = 2p \int_{\R^n} \int_{0}^{c |f(x)|} \lambda^{p-1} \Phi\left( \frac{|f(x)|}{\lambda} \right)    \, d \lambda \, d \mu(x) \\
& = 2p \int_{\R^n} \int_{1/c}^{\infty} \left( \frac{|f(x)|}{t} \right)^{p-1} \Phi(t) \left( \frac{|f(x)|}{t} \right)     \, \frac{d t}{t} \, d \mu(x)\\
& = 2p \left(\int_{\R^n} |f(x)|^p \, d \mu(x)\right) \cdot \left(\int_{1/c}^{\infty} \frac{\Phi(t)}{t^p} \frac{dt}{t}\right)\\
& = C_{p,\Phi} \int_{\R^n} |f(x)|^p \, d \mu(x),
\end{align*}
where we used the $B_p$ condition in the last step. 
\end{proof}
\end{lemma}

Fix $1<p<\infty$; we will sometimes omit the dependence on $p$ in what follows to simplify the notation. Let $W, V$ be $d \times d$ matrix weights. Given two Young functions $\Phi, \Psi$, we write 

$$[W,V]_{\Phi, \Psi}^b:= \sup_{\substack{I,J \in \mc{D} \\ 0 \leq \text{dist}_{\mc{D}}(I, J) \leq N + 2}}[c_p^b(I, J)] \| \| V^{1/p}(x) W^{-1/p}(y) \|_{\Phi_x, I}  \|_{\Psi_y, J},^p $$
where we define the (scalar) Orlicz average of a matrix weight $W$ on a dyadic cube $I$  

$$\|W(\cdot)\|_{\Phi, I}= \left\| \bv{ W(\cdot)} \right \| _{\Phi, I}$$
and recall
$$c_p^b(I,J)= \frac{m(I)^{p/2} m(J)^{p/2}}{\mu(I)^{p-1} \mu(J)}$$

Recall that for each pair of dyadic cubes $I,J$, there exist reducing operators ($d \times d$ matrices) $\mathcal{V}_{\Phi,I}$ and $\mathcal{W}_{\Psi, J}$ satisfying for each vector $e \in \R^d.$

$$ |\mathcal{V}_{\Phi, I} e| \sim \|V^{1/p} e\|_{\Phi, I}  , \quad  |\mathcal{W}_{\Psi, J} e | \sim \| W^{-1/p}(\cdot) e \|_{\Psi, J} $$
As before, we have an equivalent quantity for the two weight characteristic:

$$ [W,V]_{\Phi, \Psi}^b \sim \sup_{\substack{I,J \in \mc{D} \\ 0 \leq \text{dist}_{\mc{D}}(I, J) \leq N + 2}}[c_p^b(I, J)] \bv {\mathcal{V}_{\Phi, I} \mathcal{W}_{\Psi, J}}^p.$$

We also define the weighted maximal operators

$$M_{\Psi, W}^{\D}f(x):= \sup_{I \in \D} \left \langle  \left | \mathcal{W}_{\Psi, I}^{-1} W^{-1/p}(\cdot) f(\cdot) \right | \right \rangle_I \1_{I}(x), \quad M_{\Phi, V}^{\D}f(x):= \sup_{I \in \D} \left \langle  \left | \mathcal{V}_{\Phi, I}^{-1} V^{1/p}(\cdot) f(\cdot) \right | \right \rangle_I \1_{I}(x) .$$
Notice that we have the following bound by the generalized H\"{o}lder inequality: 

\begin{align*}
\left \langle  \left |\mathcal{W}_{\Psi, I}^{-1} W^{-1/p}(\cdot) f(\cdot) \right | \right \rangle_I & \leq  \left \langle  \bv{\mathcal{W}_{\Psi, I}^{-1} W^{-1/p}(\cdot)} \left |f(\cdot) \right | \right \rangle_I \\
& \leq 2  \|\mathcal{W}_{\Psi, I}^{-1} W^{-1/p}(\cdot)\|_{\Psi, I} \|f\|_{\overline{\Psi},I} \\
& \lesssim \|f\|_{\overline{\Psi},I}
\end{align*}
by properties of the reducing operator, so we see that we can dominate this maximal function by the scalar variant:

\begin{equation}
 M_{\Psi, W}^{\D}f(x) \lesssim M_{\overline{\Psi}}^{\D}f(x),
\label{MaximalOrliczPointwise}
 \end{equation}

 and completely analogously, we have 

\begin{equation}
 M_{\Phi, V}^{\D}f(x) \lesssim M_{\overline{\Phi}}^{\D}f(x).
 \end{equation}

\begin{theorem}
Let $T$ be a vector Haar shift, and let $\Phi, 
 \Psi$ be Young functions such that $\overline{\Phi} \in B_{p'}$ and $\overline{\Psi} \in B_{p}.$  If $W,V$ satisfy the joint Orlicz bump condition $[W,V]_{\Phi, \Psi}^{b}<\infty$, then 

$$ \|Tf \|_{L^p(V)} \lesssim  ([W,V]_{\Phi, \Psi}^b)^{1/p} \|f\|_{L^p(W)}.$$

\begin{proof}

By standard change of variable and duality, it is enough to prove for arbitrary $f \in L^p(\mu)$ and $g \in L^{p'}(\mu)$

$$ |\langle V^{1/p}T W^{-1/p}f, g \rangle| \lesssim ([W,V]_{\Phi, \Psi}^b )^{1/p}\|f\|_{L^p(\mu)} \|g\|_{L^{p'}(\mu)}.$$
By Theorem \ref{thm: sparseforvectorhaar}, it is enough to estimate the sparse form $\mathcal{A}_{\mathcal{S}}^N(W^{-1/p}f, V^{1/p}g)$, because the estimate for $\mathcal{A}_{\mathcal{S}}(W^{-1/p}f, V^{1/p}g)$ is even simpler. Proceeding similarly to the proof of Corollary \ref{cor:WeightedHaarShift}, we have

\begin{small}
\begin{align*}
 &  \sum_{\substack{J,K \in \mc{S} \\ \text{dist}_{\mc{D}}(J, K) \leq N + 2 }}\dashint_{J}\dashint_{K}|W^{-1/p}(x)f(x) \cdot V^{1/p}(y) g(y)| \, d\mu(x)d\mu(y) \sqrt{m(J)}\sqrt{m(K)} \\
  & \lesssim ([W,V]_{\Phi, \Psi}^{b})^{1/p}\left( \sum_{K \in \mathcal{S}} \LL |\mathcal{W}_{\Psi, K}^{-1} W^{-\frac{1}{p}}f| \RR_K^p \mu(E_K) \right)^{\frac{1}{p}} \left( \sum_{J \in \mathcal{S}} \LL |\mathcal{V}_{\Phi, J}^{-1} V^{\frac{1}{p}}g| \RR_J^{p'} \mu(E_J)\right)^{\frac{1}{p'}} \\
 & \lesssim ([W,V]_{\Phi, \Psi}^{b})^{1/p} \|M_{\Psi, W}^{\D} f\|_{L^p(\mu)} \|M_{\Phi, V}^{\D}g\|_{L^{p'}(\mu)} \\
 & \lesssim ([W,V]_{\Phi, \Psi}^{b})^{1/p} \|M_{\overline{\Psi}}^{\D} f\|_{L^p(\mu)} \|M_{\overline{\Phi} }^{\D}g\|_{L^{p'}(\mu)} \\
 & \lesssim ([W,V]_{\Phi, \Psi}^{b})^{1/p} \|f\|_{L^p(\mu)} \|g\|_{L^{p'}(\mu)},
\end{align*}
\end{small}
where we used Lemma \ref{MaxOrliczBound} in the last step.

\end{proof}

\end{theorem}

\bibliographystyle{alpha}
\bibliography{sources}

\begin{thebibliography}{CAPW24}

\bibitem[BPW16]{BickelPetermichlWick}
Kelly Bickel, Stefanie Petermichl, and Brett~D. Wick.
\newblock Bounds for the {H}ilbert transform with matrix {$A_2$} weights.
\newblock {\em J. Funct. Anal.}, 270(5):1719--1743, 2016.

\bibitem[CAP19]{Conde-AlonsoParcet}
Jos\'e{}~M. Conde-Alonso and Javier Parcet.
\newblock Nondoubling {C}alder\'on-{Z}ygmund theory: a dyadic approach.
\newblock {\em J. Fourier Anal. Appl.}, 25(4):1267--1292, 2019.

\bibitem[CAPW24]{CPW}
Jose Conde-Alonso, Jill Pipher, and Nathan Wagner.
\newblock Balanced measures, sparse domination and complexity-dependent weight
  classes.
\newblock {\em Mathematische Annalen}, 391:2209--2253, 2024.

\bibitem[CG01]{ChristGoldberg}
Michael Christ and Michael Goldberg.
\newblock Vector {$A_2$} weights and a {H}ardy-{L}ittlewood maximal function.
\newblock {\em Trans. Amer. Math. Soc.}, 353(5):1995--2002, 2001.

\bibitem[CT19]{CuliucTreil}
Amalia Culiuc and Sergei Treil.
\newblock The {C}arleson embedding theorem with matrix weights.
\newblock {\em Int. Math. Res. Not. IMRN}, (11):3301--3312, 2019.

\bibitem[CUIM18]{cruzuribe2017}
David Cruz-Uribe, Joshua Isralowitz, and Kabe Moen.
\newblock Two weight bump conditions for matrix weights.
\newblock {\em Integral Equations Operator Theory}, 90(3):Paper No. 36, 31,
  2018.

\bibitem[dlHH18]{HH}
Ana~Grau de~la Herr\'an and Tuomas Hyt\"onen.
\newblock Dyadic representation and boundedness of nonhomogeneous
  {C}alder\'on-{Z}ygmund operators with mild kernel regularity.
\newblock {\em Michigan Math. J.}, 67(4):757--786, 2018.

\bibitem[DPTV24]{DPTV}
Komla Domelevo, Stefanie Petermichl, Sergei Treil, and Alexander Volberg.
\newblock The matrix {$A_2$} conjecture fails, i.e. $3/2>1$, 2024.

\bibitem[Gol03]{Goldberg2003}
Michael Goldberg.
\newblock Matrix $a_p$ weights via maximal functions.
\newblock {\em Pac. J. Mat.}, 211(2):201--220, 2003.

\bibitem[H\"18]{Hanninen}
Timo~S. H\"anninen.
\newblock Equivalence of sparse and {C}arleson coefficients for general sets.
\newblock {\em Ark. Mat.}, 56(2):333--339, 2018.

\bibitem[HPTV14]{HPTV}
Tuomas Hyt\"onen, Carlos P\'erez, Sergei Treil, and Alexander Volberg.
\newblock Sharp weighted estimates for dyadic shifts and the {$A_2$}
  conjecture.
\newblock {\em J. Reine Angew. Math.}, 687:43--86, 2014.

\bibitem[Hyt12a]{hytonen2012}
Tuomas~P. Hytonen.
\newblock The a2 theorem: Remarks and complements.
\newblock 2012.

\bibitem[Hyt12b]{HytonenAnnals}
Tuomas~P. Hyt\"onen.
\newblock The sharp weighted bound for general {C}alder\'on-{Z}ygmund
  operators.
\newblock {\em Ann. of Math. (2)}, 175(3):1473--1506, 2012.

\bibitem[Lac17]{Lacey_2017}
Michael~T. Lacey.
\newblock An elementary proof of the {$A_2$} bound.
\newblock {\em Israel Journal of Mathematics}, 217(1):181–195, March 2017.

\bibitem[Ler13]{Lerner2013}
Andrei~K. Lerner.
\newblock A simple proof of the {$A_2$} conjecture.
\newblock {\em Int. Math. Res. Not. IMRN}, (14):3159--3170, 2013.

\bibitem[LSMP14]{LSMP}
Luis~Daniel L{\'o}pez-S{\'a}nchez, Jos{\'e}~Mar{\'\i}a Martell, and Javier
  Parcet.
\newblock Dyadic harmonic analysis beyond doubling measures.
\newblock {\em Advances in Mathematics}, 267:44--93, 2014.

\bibitem[Moe12]{Moen}
Kabe Moen.
\newblock Sharp weighted bounds without testing or extrapolation.
\newblock {\em Arch. Math. (Basel)}, 99(5):457--466, 2012.

\bibitem[NPTV17]{NPTV}
Fedor Nazarov, Stefanie Petermichl, Sergei Treil, and Alexander Volberg.
\newblock Convex body domination and weighted estimates with matrix weights.
\newblock {\em Adv. Math.}, 318:279--306, 2017.

\bibitem[NTV97]{NTV97}
F.~Nazarov, S.~Treil, and A.~Volberg.
\newblock Cauchy integral and {C}alder\'on-{Z}ygmund operators on
  nonhomogeneous spaces.
\newblock {\em Internat. Math. Res. Notices}, (15):703--726, 1997.

\bibitem[NTV98]{NTV98}
F.~Nazarov, S.~Treil, and A.~Volberg.
\newblock Weak type estimates and {C}otlar inequalities for
  {C}alder\'on-{Z}ygmund operators on nonhomogeneous spaces.
\newblock {\em Internat. Math. Res. Notices}, (9):463--487, 1998.

\bibitem[NTV03]{NTV2003}
F.~Nazarov, S.~Treil, and A.~Volberg.
\newblock The {$Tb$}-theorem on non-homogeneous spaces.
\newblock {\em Acta Math.}, 190(2):151--239, 2003.

\bibitem[P\'95]{Perez1995}
C.~P\'erez.
\newblock On sufficient conditions for the boundedness of the
  {H}ardy-{L}ittlewood maximal operator between weighted {$L^p$}-spaces with
  different weights.
\newblock {\em Proc. London Math. Soc. (3)}, 71(1):135--157, 1995.

\bibitem[Pet07]{Petermichl}
S.~Petermichl.
\newblock The sharp bound for the {H}ilbert transform on weighted {L}ebesgue
  spaces in terms of the classical {$A_p$} characteristic.
\newblock {\em Amer. J. Math.}, 129(5):1355--1375, 2007.

\bibitem[PV02]{PetermichlVolberg}
Stefanie Petermichl and Alexander Volberg.
\newblock Heating of the {A}hlfors-{B}eurling operator: weakly quasiregular
  maps on the plane are quasiregular.
\newblock {\em Duke Math. J.}, 112(2):281--305, 2002.

\bibitem[Tol01a]{TolsaRBMO}
Xavier Tolsa.
\newblock B{MO}, {$H^1$}, and {C}alder\'on-{Z}ygmund operators for non doubling
  measures.
\newblock {\em Math. Ann.}, 319(1):89--149, 2001.

\bibitem[Tol01b]{TolsaT1}
Xavier Tolsa.
\newblock Littlewood-{P}aley theory and the {$T(1)$} theorem with non-doubling
  measures.
\newblock {\em Adv. Math.}, 164(1):57--116, 2001.

\bibitem[Tol07]{Tolsa2007}
Xavier Tolsa.
\newblock Weighted norm inequalities for {C}alder\'on-{Z}ygmund operators
  without doubling conditions.
\newblock {\em Publ. Mat.}, 51(2):397--456, 2007.

\bibitem[Tre23]{Treil2023}
Sergei Treil.
\newblock Mixed {$A_2$}-{$A_\infty$} estimates of the non-homogeneous vector
  square function with matrix weights.
\newblock {\em Proc. Amer. Math. Soc.}, 151(8):3381--3389, 2023.

\bibitem[TTV15]{ThieleTreilVolberg}
Christoph Thiele, Sergei Treil, and Alexander Volberg.
\newblock Weighted martingale multipliers in the non-homogeneous setting and
  outer measure spaces.
\newblock {\em Adv. Math.}, 285:1155--1188, 2015.

\bibitem[TV97]{TreilVolberg}
S.~Treil and A.~Volberg.
\newblock Wavelets and the angle between past and future.
\newblock {\em J. Funct. Anal.}, 143(2):269--308, 1997.

\bibitem[Vol97]{Volberg97}
A.~Volberg.
\newblock Matrix {$A_p$} weights via {$S$}-functions.
\newblock {\em J. Amer. Math. Soc.}, 10(2):445--466, 1997.

\bibitem[VV20]{Vasyunin_Volberg_2020}
Vasily Vasyunin and Alexander Volberg.
\newblock {\em The Bellman Function Technique in Harmonic Analysis}.
\newblock Cambridge Studies in Advanced Mathematics. Cambridge University
  Press, 2020.

\bibitem[VZK18]{VolbergZorin-Kranich}
Alexander Volberg and Pavel Zorin-Kranich.
\newblock Sparse domination on non-homogeneous spaces with an application to
  {$A_p$} weights.
\newblock {\em Rev. Mat. Iberoam.}, 34(3):1401--1414, 2018.

\end{thebibliography}

\end{document}